\documentclass[12pt]{article}

\usepackage{amsmath,amsthm,amsfonts,amssymb}
\usepackage{tikz}
\usetikzlibrary{positioning}
\tikzset{mynode/.style={draw,circle,inner sep=2pt,outer sep=0pt}}

\theoremstyle{plain}
\newtheorem{theorem}{Theorem}[section]
\newtheorem{proposition}[theorem]{Proposition}
\newtheorem{lemma}[theorem]{Lemma}
\newtheorem{corollary}[theorem]{Corollary}

\theoremstyle{definition}
\newtheorem{definition}[theorem]{\it Definition}
\newtheorem{algorithm}[theorem]{\it Algorithm}

\theoremstyle{remark}
\newtheorem{remark}[theorem]{Remark}

\begin{document}

\title{A complete characterization of monotonicity equivalence for continuous-time Markov processes}

\author{
Motoya Machida \\
{\normalsize Department of Mathematics, Tennessee Technological University} \\
{\normalsize Cookeville, TN 38505, USA}
}

\date{\today}

\maketitle

\begin{abstract}
Dai Pra~{et al.}\ studied two notions of monotonicity for continuous-time Markov
processes on a finite partially ordered set (poset), and conjectured that monotonicity
equivalence holds for a poset of W-glued diamond, and that
there is no other case when it has no acyclic extension.
We proved their conjecture and were able to provide a complete
characterization of posets for
monotonicity equivalence.

\medskip\par\noindent
{\em AMS\/} 2020 {\em subject classifications.\/}
Primary 60J27;
secondary 60E05, 06A06.


\end{abstract}

\section{Introduction}
\setcounter{equation}{0}\setcounter{figure}{0}

Let $\mathcal{A}$ and $\mathcal{S}$
be finite partially ordered sets (posets),
and let $P=(P_\alpha:\alpha\in\mathcal{A})$ be a system
of probability measures on $\mathcal{S}$.
In this paper the symbol $P_\alpha$ for a measure is interchangeably used for a
probability mass function (pmf).
A subset $U$ of $\mathcal{S}$ is called an \emph{up-set} (or \emph{increasing set})
if $x \in U$ and $x \le y$ in $\mathcal{S}$ imply $y \in U$.
For a pair $(P_\alpha, P_\beta)$ of probability measures
we say that $P_\alpha$ is \emph{stochastically smaller} than $P_\beta$,
denoted by $P_\alpha \preceq P_\beta$,
if $P_\alpha(U) \le P_\beta(U)$ holds for every up-set $U$ in $\mathcal{S}$,
We call the system $P$ \emph{stochastically monotone}
if the stochastic ordering is consistent with the ordering of $\mathcal{A}$,
that is, if $P_\alpha \preceq P_\beta$ whenever $\alpha\le\beta$ in $\mathcal{A}$.
Consider a collection $(\xi(\alpha): \alpha\in\mathcal{A})$
of $S$-valued random variables
on some probability space $(\Omega,\mathcal{F},\mathbb{P})$.
Then it can be viewed as a random map $\xi$ from $\mathcal{A}$
to $\mathcal{S}$,
and it realizes the system $P$ if it satisfies
${\mathbb{P}(\xi(\alpha) = x)} = P_\alpha(x)$ for $\alpha\in\mathcal{A}$.
The random map $\xi$ is said to be \emph{monotone}
(or \emph{increasing}) if $\xi(\alpha) \le \xi(\beta)$ in $\mathcal{S}$
whenever $\alpha\le\beta$ in $\mathcal{A}$.
The system $P$ is called \emph{realizably monotone}
if it is realized by such a random monotone map $\xi$.
In general realizable monotonicity is sufficient but not necessary
for stochastic monotonicity.

In \cite{fm2001,fm2002,classw}
we examined these notions of monotonicity
and characterized pairs $(\mathcal{A},\mathcal{S})$ of posets
for which monotonicity equivalence holds.
In~\cite{fm2001} we consider the case for $\mathcal{A} = \mathcal{S}$,
view the system $P=(P(x,\cdot): x\in\mathcal{S})$
as a transition probability $P(x,y)$ of
discrete-time Markov chain on $\mathcal{S}$,
and obtain

\begin{proposition}\label{fm.claim}
Monotonicity equivalence holds for a discrete-time Markov chain on $\mathcal{S}$
if and only if $\mathcal{S}$ is acyclic
(see Section~\ref{poset.subclass} for definition of an acyclic poset).
\end{proposition}

Dai Pra~{et al.}~\cite{pra}
proposed the notion of realizable monotonicity
for continuous-time Markov processes,
and studied monotonicity equivalence problems for various posets.
Their investigation uses a generator representation extensively,
and yields examples of non-acyclic poset
for which two notions of monotonicity are equivalent,
while it produces many other examples of monotonicity inequivalence.
In Section~\ref{mc.monotone} we introduce notions of monotonicity for
continuous-time Markov processes and correspond them to the notions
of weak monotonicity for discrete-time Markov chains in the sense
of Proposition~\ref{weak.equivalence},
by which the case of Proposition~\ref{main.claim}(i)
is immediately verified.
In Section~\ref{acyclic.extension.section}
we solve a monotonicity equivalence problem
[i.e., the case for Proposition~\ref{main.claim}(ii)]
when the poset has an acyclic extension.
In Section~\ref{pra.conjecture.section}
we examine it when the poset has no acyclic extension,
and complete the proof of a conjecture by~\cite{pra}
[i.e., the case for Proposition~\ref{main.claim}(iii)].
Consequently we are able to obtain a full characterization of finite posets
for continuous-time Markov processes to possess monotonicity equivalence.

\begin{proposition}\label{main.claim}
Monotonicity equivalence holds for a continuous-time Markov process on $\mathcal{S}$
if and only if $\mathcal{S}$ is
(i) acyclic,
(ii) a Y-glued bipartite (Definition~\ref{y-glued.bipartite}),
or
(iii) a W-glued diamond (Definition~\ref{w-glued.diamond}).
\end{proposition}

In \cite{classw} we developed a notion of recursive inverse transform for a
poset $\mathcal{S}$ of W-class, and used it to examine monotonicity
equivalence problems
(e.g., Corollary~\ref{w.cor}), which we briefly overview in Section~\ref{rit}.
In section~\ref{w.sec} we demonstrate the sufficiency of stochastic
monotonicity by applying a modification of recursive inverse transform when 
$\mathcal{S}$ is a poset of W-glued diamond.
We found the same constructive technique applicable for a celebrated result of
Nachbin-Strassen theorem when a poset of W-glued diamond is considered,
and discuss it in Section~\ref{strassen}.

\subsection{Poset subclasses}
\label{poset.subclass}

Throughout the present investigation
we treat a Hasse diagram as an undirected graph,
and use a calligraphic letter $\mathcal{S}$ for a poset
and an italic letter $S$ for a Hasse diagram.
Here the dual of a poset is the poset equipped with the reversed ordering.
We refer a poset and its dual to the same Hasse diagram,
and assume that all the posets are connected
(i.e., the Hasse diagram is connected).
A poset is called \emph{non-acyclic} if the Hasse diagram has a cycle;
otherwise, it is called \emph{acyclic}.
An induced subposet of $\mathcal{S}$ is the poset on a subset $S'$
equipped with the ordering of $\mathcal{S}$ restricted on $S'$.
A non-acyclic poset is said to have an \emph{acyclic extension}
if it is an induced subposet of some acyclic poset.

\begin{figure}[h]
\begin{center}
\begin{tabular}{ccccccc}
    
\begin{tikzpicture}
  \node [mynode,label=above:$d$] (d) at (0,0) {};
  \node [mynode,label=right:$c$] (c) at (1,-1) {};
  \node [mynode,label=left:$b$] (b) at (-1,-1) {};
  \node [mynode,label=below:$a$] (a) at (0,-2) {};
  \draw
  (d) -- (c)
  (d) -- (b)
  (c) -- (a)
  (b) -- (a);
\end{tikzpicture}

& \hspace{0.1in} &

\begin{tikzpicture}
  \node [mynode,label=above:$g$] (g) at (0,0) {};
  \node [mynode,label=above:$h$] (h) at (1.6,0) {};
  \node [mynode,label=below:$e$] (e) at (0,-1.6) {};
  \node [mynode,label=below:$f$] (f) at (1.6,-1.6) {};
  \draw
  (g) -- (e)
  (g) -- (f)
  (h) -- (e)
  (h) -- (f);
\end{tikzpicture}

& \hspace{0.1in} &

\begin{tikzpicture}
  \node [mynode,label=above:$g$] (g) at (0,0) {};
  \node [mynode,label=above:$h$] (h) at (1.6,0) {};
  \node [mynode,label=right:$f$] (f) at (0.8,-1) {};
  \node [mynode,label=below:$e$] (e) at (0.8,-2) {};
  \draw
  (g) -- (f)
  (h) -- (f)
  (f) -- (e);
\end{tikzpicture}

& \hspace{0.1in} &

\begin{tikzpicture}
  \node [mynode,label=above:$f$] (f) at (0,0) {};
  \node [mynode,label=above:$g$] (g) at (1,0) {};
  \node [mynode,label=above:$h$] (h) at (2,0) {};
  \node [mynode,label=below:$e$] (e) at (1,-1.6) {};
  \draw
  (f) -- (e)
  (g) -- (e)
  (h) -- (e);
\end{tikzpicture}

\\
(i) Diamond
& & (ii) Bowtie
& & (iii) Y-poset
& & (iv) W-poset
\end{tabular}
\caption{Hasse diagrams of four-element poset}
\label{named.posets}
\end{center}
\end{figure}
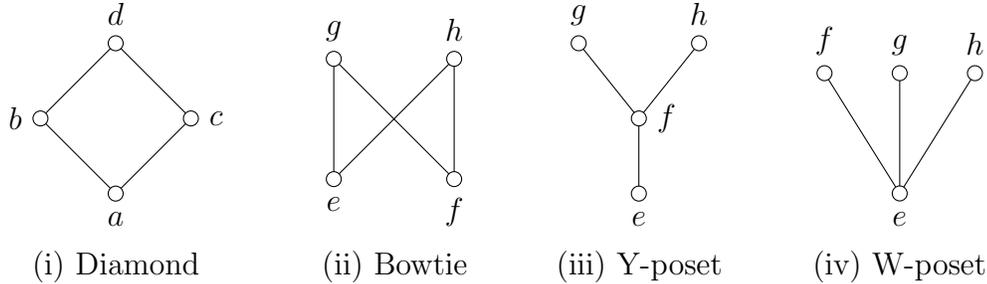

An introduction of poset subclasses
starts with Hasse diagrams of four-element poset in Figure~\ref{named.posets},
namely (i) diamond, (ii) bowtie, (iii) Y-poset, and (iv) W-poset.
A diamond and a bowtie are non-acyclic,
and the bowtie has an acyclic extension.
We say that a poset is in \emph{Y-class} if it is acyclic
and it has no bowtie as an induced subposet;
it is in \emph{W-class} if it is in Y-class
and it has neither Y-poset nor its dual as an induced subposet.
A Y-poset and a W-poset of Figure~\ref{named.posets}(iii) and (iv)
belong to Y-class, but the W-poset and its dual are in W-class.

\subsection{Monotonicity for continuous-time Markov processes}
\label{mc.monotone}

A continuous-time Markov process on $\mathcal{S}$
is characterized by the transition probability $P_t$, $t\ge 0$,
and it is obtained as
the unique solution to the Kolmogorov's forward equation
$\frac{d}{dt}P_t(x,y) = \sum_{z\in S} P_t(x,z) L(z,y)$.
The generator $L$ is determined by nonnegative values $L(x,y)$, $x\neq y$,
and satisfies $L(x,x) = -\sum_{y\in S\setminus\{x\}} L(x,y)$.
It can be further characterized by the underlying transition probability $Q_\lambda$ of jumps
(cf. \c{C}inlar~\cite{cinlar}).
Since $\mathcal{S}$ is finite,
we can set $\lambda_* = \max_{x \in S}[-L(x,x)]$
and assume $0 < \lambda_* < \infty$.
We fix $\lambda \ge \lambda_*$ arbitrarily,
and construct a transition probability
\begin{equation}\label{q.lambda}
  Q_\lambda(x,y) = I(x,y) + \dfrac{1}{\lambda} L(x,y),
  \quad x,y\in\mathcal{S},
\end{equation}
where $I(x,x) = 1$ and $I(x,y) = 0$ if $x \neq y$.
Let $N_t$ be a Poisson process with rate $\lambda$,
and let $\hat{X}_n$, $n=0,1,2,\ldots$, be a sample path
of discrete-time Markov chain generated by $Q_\lambda$.
Then the stochastic process
\begin{equation}\label{xi}
 X_t = \hat{X}_{N_t},
 \quad t \ge 0,
\end{equation}
is a sample path of continuous-time Markov process,
and satisfies
$P_t(x,y) = \mathbb{P}(X_t=y\,|\,X_0=x)$.

We call $P_t$ \emph{stochastically monotone}
if $P_t(x,\cdot) \preceq P_t(y,\cdot)$ for all $t \ge 0$
whenever $x \le y$ in $\mathcal{S}$.
In Lemma~\ref{sm.lem}
we write $L(x,U) = \sum_{z\in U} L(x,z)$ for a subset $U$ of $\mathcal{S}$.

\begin{lemma}\label{sm.lem}
Each of the following conditions is equivalent to the stochastic monotonicity
for $P_t$.  
\begin{enumerate}  
\renewcommand{\labelenumi}{(\alph{enumi})}
\item
Whenever $x \le y$ in $\mathcal{S}$
we have (i) $L(x,U) \le L(y,U)$ for each up-set $U$ satisfying $y \not\in U$,
and (ii) $L(x,U^c) \ge L(y,U^c)$ for each up-set $U$ satisfying $x \in U$.
\item
$Q_\lambda$ is stochastically monotone
for some $\lambda \ge \lambda_*$.
\end{enumerate}
\end{lemma}

\begin{proof}
The equivalent characterization (a) of stochastic monotonicity was obtained by
Massey~\cite{massey}.
Suppose that (b) holds.
Then if $x \le y$ in $\mathcal{S}$
then $Q_\lambda$ generates a pair $\hat{X}_n \le \hat{Y}_n$
of sample paths with initial state $\hat{X}_0 = x$ and $\hat{Y}_0 = y$.
The resulting pair of continuous-time Markov processes of \eqref{xi}
respectively with $\hat{X}_n$ and $\hat{Y}_n$
implies that $P_t(x,\cdot) \preceq P_t(y,\cdot)$,
and therefore, $P_t$ is stochastically monotone.

Suppose that (a) holds.
Then we set $\lambda = 2 \lambda_*$.
Let $x \le y$ and let $U$ be an up-set in $\mathcal{S}$.
If $x \in U$ or $y \not\in U$ then
we can easily verify
$Q_\lambda(x,U) \le Q_\lambda(y,U)$.
If $x \not\in U$ and $y \in U$
then
$$
  Q_\lambda(x,U) = \frac{1}{\lambda} L(x,U)
  \le \frac{1}{2} \le 1 + \frac{1}{\lambda} L(y,U)
  = Q_\lambda(y,U)
$$
Thus, $Q_\lambda$ is stochastically monotone.
\end{proof}

Let $\mathcal{M}$ be the space of monotone maps from $\mathcal{S}$ to itself
(i.e., $h\in\mathcal{M}$ implies $h(x) \le h(y)$ whenever $x \le y$
in $\mathcal{S}$).
By $\mathrm{id}\in\mathcal{M}$ we denote the identity map
[i.e., $\mathrm{id}(x)=x$ for all $x\in\mathcal{S}$].
If there exists a sample path $\xi_t$ of continuous-time Markov process
on the state space $\mathcal{M}$ starting from $\xi_0=\mathrm{id}$
and marginally distributed as $P_t$
[i.e., $\mathbb{P}(\xi_t(x) = y\,|\,\xi_0(x)=x) = P_t(x,y)$]
then $P_t$ is said to be \emph{realizably monotone}.

\begin{lemma}\label{rm.lem}
Each of the following conditions is equivalent to the realizable monotonicity
for $P_t$.  
\begin{enumerate}  
\renewcommand{\labelenumi}{(\alph{enumi})}
\item
There exists a nonnegative function $\gamma(h)$ on
$\tilde{\mathcal{M}}=\mathcal{M}\setminus\{\mathrm{id}\}$ such that
\begin{equation}\label{l.margin}
  L(x,y) = \sum_{h\in\tilde{\mathcal{M}}} \gamma(h) I(h(x),y)
\end{equation}
for all $x \neq y$.
\item
$Q_\lambda$ is realizably monotone
for some $\lambda \ge \lambda_*$.
\end{enumerate}  
\end{lemma}

\begin{proof}
The equivalent characterization (a) of realizable monotonicity
was obtained by Proposition~2.1 of~\cite{pra}.
Suppose that (b) holds.
Then we can construct
a sample path $\hat{\xi}_n$ of discrete-time Markov chain
on the state space $\mathcal{M}$ starting from $\hat{\xi}_0=\mathrm{id}$
and marginally distributed as $Q_\lambda$
[i.e., $\mathbb{P}(\hat{\xi}_n(x) = y\,|\,\hat{\xi}_0(x)=x) = Q_\lambda^n(x,y)$].
Thus, $P_t$ must be realizably monotone via~(\ref{xi}).

Suppose that (a) holds.
Then we set $\lambda = \sum_{h\in\tilde{\mathcal{M}}} \gamma(h)$
and obtain $\lambda \ge \lambda_*$ so that $\gamma(\cdot)/\lambda$
becomes a pmf on $\tilde{\mathcal{M}}$.
By \eqref{l.margin}
we have $\sum_{h\in\tilde{\mathcal{M}}} (\gamma(h)/\lambda) I(h(x),y) = Q_\lambda(x,y)$
for $x \neq y$.
If $x = y$ then we find
\begin{equation*}
  \sum_{h\in\tilde{\mathcal{M}}} \frac{\gamma(h)}{\lambda} I(h(x),x)
  = 1 - \left(\frac{1}{\lambda}\right) \hspace{-0.1in}
   \sum_{h\in\tilde{\mathcal{M}}:\,h(x)\neq x}
   \hspace{-0.1in}\gamma(h)
  = 1 + \frac{1}{\lambda} L(x,x) = Q_\lambda(x,x)
\end{equation*}
Therefore, $Q_\lambda$ is realizably monotone,
and the pmf $\gamma(\cdot)/\lambda$
generates a desired random monotone map $\hat{\xi}$
[i.e., $\mathbb{P}(\hat{\xi} = h) = \gamma(h)/\lambda$
for $h\in\tilde{\mathcal{M}}$].
\end{proof}

In the rest of investigation we consider an induced subposet $\mathcal{A}$ of $\mathcal{S}$.
By $I_x$ we denote the unit mass at $x$
[i.e., $I_x(U) = 1$ if $x \in U$; otherwise, $I_x(U) = 0$],
and by $I=(I_\alpha:\alpha\in\mathcal{A})$ the system of unit masses.
Then we call a system $P=(P_\alpha: \alpha\in\mathcal{A})$ stochastically
(respectively realizably) \emph{weakly monotone}
if a linear combination
\begin{equation*}
  \theta P + (1-\theta) I
  = (\theta P_\alpha + (1-\theta) I_\alpha: \alpha\in\mathcal{A})
\end{equation*}
is stochastically (respectively realizably) monotone for some $\theta\in (0,1]$.
We say that weak monotonicity equivalence holds
for a pair $(\mathcal{A},\mathcal{S})$ of posets
if every stochastically weakly monotone system $P$
is realizably weakly monotone for the pair.

In particular when $\mathcal{A}=\mathcal{S}$
we can observe for $\lambda\ge\lambda_0\ge\lambda_*$ that
the transition probability $Q_\lambda$ of \eqref{q.lambda} can be viewed as
the linear combination $\theta Q_{\lambda_0} + (1-\theta) I$
with $\theta=\lambda_0/\lambda$.
Hence, Lemma~\ref{sm.lem} and~\ref{rm.lem} has established
the following characterization of monotonicity equivalence
for continuous-time Markov processes
(Proposition~2.7 of~\cite{pra}).

\begin{proposition}\label{weak.equivalence}
Monotonicity equivalence holds for a continuous-time Markov process
on $\mathcal{S}$
if and only if
weak monotonicity equivalence holds for a discrete-time Markov chain
on $\mathcal{S}$.
\end{proposition}

Since the identity transition probability $I$ is trivially realizably monotone,
a transition probability is stochastically (or realizably) weakly monotone
if it is stochastically (or realizably) monotone.
Therefore, Proposition~\ref{fm.claim}
immediately indicates monotonicity equivalence for
the case of Proposition~\ref{main.claim}(i).

\section{A subclass of posets with acyclic extension}
\label{acyclic.extension.section}
\setcounter{equation}{0}\setcounter{figure}{0}

In this section we consider
a subclass of non-acyclic (and connected) posets having acyclic extension.
If a pair $(x,y)$ is an edge of the Hasse diagram $S$ and $x < y$ in $\mathcal{S}$
then $y$ is said to \emph{cover} $x$.
A non-acyclic poset $\mathcal{A}$ of Figure~\ref{bipartite}(i) is called
$(m,n)$-bipartite
if the Hasse diagram $A$ is complete bipartite with $m,n\ge 2$.
Starting from a poset $\mathcal{S}$ of this subclass
we can construct a pair $\mathcal{S}_1$ and $\mathcal{S}_2$ of posets in
Algorithm~\ref{ext.alg}.

\begin{algorithm}\label{ext.alg}
Since the Hasse diagram $S$ contains a bowtie of Figure~\ref{named.posets}(ii)
(Proposition~5.11 of~\cite{fm2001}) as a subgraph of $S$,
we can find a maximal bipartite subgraph $A$ of $S$
which induces an $(m,n)$-bipartite poset $\mathcal{A}$ of Figure~\ref{bipartite}(i).
By removing all the edges $(a_i,b_j)$ in $A$
and inserting a new vertex $c$ along with the edges $(a_i,c)$ and $(c,b_j)$,
we create an extension $\hat{\mathcal{S}}$
where $b_j$ covers $c$ and $c$ covers $a_i$.
We can view $\mathcal{S}$ as the induced subposet of $\hat{\mathcal{S}}$.
We can also construct
the induced subposet $\mathcal{S}_1$ consisting of vertices
equal to $c$ or connected to $c$ via some $a_i$ in $\hat{S}$,
and the induced subposet $\mathcal{S}_2$ consisting of vertices
equal to $c$ or connected to $c$ via some $b_j$ in $\hat{S}$.
\end{algorithm}

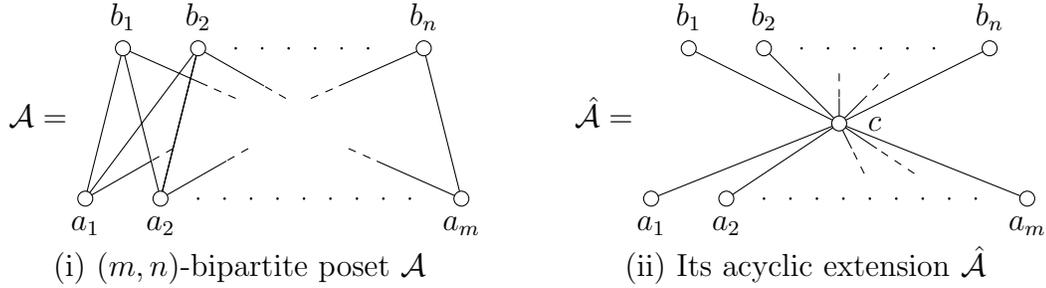
\begin{figure}[h]
\begin{center}
\begin{tabular}{ccc}

$\mathcal{A} =$
\hspace{-2ex}  
\begin{minipage}{30ex}  
\begin{tikzpicture}
  \tikzset{circle dotted/.style={line cap=round, line width=1pt, dash pattern=on 0pt off 10pt}}
  \node [mynode,label=above:$b_1$] (b1) at (0,0) {};
  \node [mynode,label=above:$b_2$] (b2) at (1,0) {};
  \node [mynode,label=above:$b_n$] (bn) at (4,0) {};
  \node [mynode,label=below:$a_1$] (a1) at (-0.5,-2) {};
  \node [mynode,label=below:$a_2$] (a2) at (0.5,-2) {};
  \node [mynode,label=below:$a_m$] (am) at (4.5,-2) {};
  \draw
  (b1) -- (a1)
  (b1) -- (a2)
  (b2) -- (a1)
  (b2) -- (a2)
  (b2) -- (a2)
  (bn) -- (am)
  (b1) -- (4.5/4,-2/4)
  (b2) -- (1+3.5/4,-2/4)
  (bn) -- (4-4.5/4,-2/4)
  (a1) -- (-0.5+3.5/4,-2+2/4)
  (a2) -- (0.5+3.5/4,-2+2/4)
  (am) -- (4.5-4.5/4,-2+2/4);
  \draw [dashed]
  (4.5/3,-2/3) -- (4.5/4,-2/4)
  (1+3.5/3,-2/3) -- (1+3.5/4,-2/4)
  (4-4.5/3,-2/3) -- (4-4.5/4,-2/4)
  (-0.5+3.5/3,-2+2/3) -- (-0.5+3.5/4,-2+2/4)
  (0.5+3.5/3,-2+2/3) -- (0.5+3.5/4,-2+2/4)
  (4.5-4.5/3,-2+2/3) -- (4.5-4.5/4,-2+2/4);
  \draw [circle dotted]
  (1.5,0) -- (3.5,0)
  (1.0,-2) -- (4.0,-2);
\end{tikzpicture}
\end{minipage}

& \hspace{3ex} &

$\hat{\mathcal{A}} =$
\hspace{-2ex}  
\begin{minipage}{30ex}  
\begin{tikzpicture}
  \tikzset{circle dotted/.style={line cap=round, line width=1pt, dash pattern=on 0pt off 10pt}}
  \node [mynode,label=above:$b_1$] (b1) at (0,0) {};
  \node [mynode,label=above:$b_2$] (b2) at (1,0) {};
  \node [mynode,label=above:$b_n$] (bn) at (4,0) {};
  \node [mynode,label=below:$a_1$] (a1) at (-0.5,-2) {};
  \node [mynode,label=below:$a_2$] (a2) at (0.5,-2) {};
  \node [mynode,label=below:$a_m$] (am) at (4.5,-2) {};
  \node [mynode,label=right:$\ c$] (c) at (2,-1) {};
  \draw
  (b1) -- (c)
  (b2) -- (c)
  (bn) -- (c)
  (a1) -- (c)
  (a2) -- (c)
  (am) -- (c)
  (c) -- (2+1/3,-1+1/3)
  (c) -- (2+0/3,-1+1/3)
  (c) -- (2+1.5/3,-1-1/3)
  (c) -- (2+0.5/3,-1-1/3);
  \draw [dashed]
  (2+1/1.5,-1+1/1.5) -- (2+1/3,-1+1/3)
  (2+0/1.5,-1+1/1.5) -- (2+0/3,-1+1/3)
  (2+1.5/1.5,-1-1/1.5) -- (2+1.5/3,-1-1/3)
  (2+0.5/1.5,-1-1/1.5) -- (2+0.5/3,-1-1/3);
  \draw [circle dotted]
  (1.5,0) -- (3.5,0)
  (1.0,-2) -- (4.0,-2);
\end{tikzpicture}
\end{minipage}
\\
(i) $(m,n)$-bipartite poset $\mathcal{A}$
& & (ii) Its acyclic extension $\hat{\mathcal{A}}$
\end{tabular}
\caption{Bipartite poset and its acyclic extension}
\label{bipartite}
\end{center}
\end{figure}

When two posets $\mathcal{S}_1$ and $\mathcal{S}_2$
share a common vertex $\{c\} = S_1\cap S_2$,
they can be glued at $c$.
The resulting new poset $\mathcal{S}$ has
the Hasse diagram $S$ with all the vertices and edges from $S_1$ and $S_2$
in which the common vertex $c$ is shared.
In Algorithm~\ref{ext.alg}
we can view $\hat{\mathcal{S}}$ as the poset obtained from
$\mathcal{S}_1$ and $\mathcal{S}_2$ glued at $c$.
We can also observe that
(a)
the induced subposet $\hat{\mathcal{A}}$ of $\hat{\mathcal{S}}$
on $A\cup\{c\}$ becomes the acyclic extension of $\mathcal{A}$
as in Figure~\ref{bipartite}(ii),
and (b)
both $\mathcal{S}_1$ and $\mathcal{S}_2$ have an acyclic extension
(Lemma~5.12 of~\cite{fm2001}).

\begin{lemma}\label{y.prop.necessity}
Let $(\mathcal{S}_1, \mathcal{S}_2)$ be a pair of posets
constructed by Algorithm~\ref{ext.alg}.
Then weak monotonicity equivalence fails for a discrete-time Markov chain on
$\mathcal{S}$
if either $\mathcal{S}_1$ or $\mathcal{S}_2$ is not in Y-class.
\end{lemma}

\begin{proof}
Suppose that $\mathcal{S}_1$ is not in Y-class.
Even though $\mathcal{S}_1$ could be acyclic,
it contains an induced bowtie poset $\mathcal{B}$ of $\mathcal{S}$
labeled with $B=\{e,f,g,h\}$ as in Figure~\ref{named.posets}(ii)
satisfying $e,f,g \in S_1\setminus\{c\}$.
We can choose it so that $h\in S_2\setminus\{c\}$ whenever possible.
If not, the induced subgraph of $\hat{S}$ on $S_1\setminus\{c\}$ becomes a forest
consisting of $m$ trees; otherwise, there exists a pair $a_i$ and $a_j$ which are connected
in $S_1\setminus\{c\}$, implying an induced bowtie of $\mathcal{S}_1$ with $h=c$.
Consequently, the bowtie $B$ is contained in one of the trees,
say, connected to $a_2$,
and $a_1$ is incomparable with $e,f,g,h$.
If $h\in S_2\setminus\{c\}$ then we can
stipulate that $h = b_2$.
Then we can observe that $g$ cannot be less than or equal to any of $a_i$'s because $g$ and
$h = b_2$ are incomparable.
If $g$ is greater than all of $a_i$'s in $\mathcal{S}$
then we must find a bipartite subgraph strictly larger than $A$.
Since the maximal $A$ has been chosen among bipartite subgraphs,
we can assume that $g$ is not comparable with $a_1$.
As a result, $a_1$ cannot be less than or equal to $e$ nor $f$.
Let $A' = \{a_1,a_2,b_1,b_2\}$ be a bowtie of $S$,
and let
\begin{math}
  S' = \{x \in S\setminus A':
  \mbox{$x \ge a_1$ or $x \ge a_2$ in $\mathcal{S}$}
  \}
\end{math}
be an up-set of $\mathcal{S}$.
Here we can construct a stochastically monotone transition probability
$(P_x: x\in\mathcal{S})$ of discrete-time Markov chain by setting
\begin{equation*}
  P_{a_1} := \frac{1}{3} I_{\{e,f,h\}}; \quad
  P_{a_2} := \frac{1}{3} I_{\{e,f,g\}}; \quad
  P_{b_1} := \frac{1}{3} I_{\{f,g,h\}}; \quad
  P_{b_2} := \frac{1}{3} I_{\{e,g,h\}};
\end{equation*}
and
\begin{equation*}
P_x = \begin{cases}
  \frac{1}{2}I_{\{g,h\}} & \mbox{ if $x \in S'$; } \\
  \frac{1}{2}I_{\{e,f\}} & \mbox{ if $x \not\in S'\cup A'$, }
\end{cases}
\end{equation*}
where $I_U$ is the indicator function on a subset~$U$
[i.e., $I_U(x) = 1$ if $x\in U$; otherwise, $I_U(x) = 0$].

Suppose that a random monotone map $\xi_x$
realizes $\tilde{P}_x = \theta P_x + (1-\theta)I_x$, $x\in S$,
for some $\theta>0$.
Since $b_2\neq e$ and $a_1\not\le e$, we obtain $\mathbb{P}(\xi_{b_2} = e,\,\xi_{a_1} = e) = \theta/3$.
Note that $\tilde{P}_{a_2}(g)$ could become $(1-(2/3)\theta)$ if $h\not\in S_2\setminus\{c\}$,
and that $\mathbb{P}(\xi_{a_2} = g,\,\xi_{b_2} = g) = \theta/3$ in either of
the cases for $h$.
Since $a_1\neq g$ and
the event $\{\xi_{a_1}=e\}$ cannot happen when $\xi_{b_2}=g$,
we conclude
$\mathbb{P}(\xi_{a_2} = g,\,\xi_{b_2} = g,\,\xi_{a_1} = f) = \theta/3$.
Since $b_1\neq f$ and $a_1\not\le f$,
we have
$\mathbb{P}(\xi_{b_1} = f,\,\xi_{a_1} = f,\,\xi_{a_2} \le f) = \theta/3$.
Since $\{\xi_{a_2}=g\}$ and $\{\xi_{a_2}\le f\}$ are disjoint,
we obtain
$$
\mathbb{P}(\xi_{a_1} = f)
\ge \mathbb{P}(\xi_{a_2} = g,\,\xi_{b_2} = g,\,\xi_{a_1} = f)
+ \mathbb{P}(\xi_{b_1} = f,\,\xi_{a_1} = f,\,\xi_{a_2} \le f)
= \frac{2}{3}\theta,
$$
which contradict that
$\tilde{P}_{a_1}(f) = \theta/3$.
Therefore, the transition probability $(P_x: x\in\mathcal{S})$ cannot be realizably weakly monotone.
The case that $\mathcal{S}_2$ is not in Y-class
can be dually handled.
\end{proof}

Lemma~\ref{y.prop.necessity} implies
the necessity of Definition~\ref{y-glued.bipartite}
for a poset to attain weak monotonicity equivalence.
In the rest of this section we present
the proof of sufficiency for Proposition~\ref{y.prop},
and consequently obtain monotonicity equivalence for Proposition~\ref{main.claim}(ii).

\begin{definition}\label{y-glued.bipartite}
We call $\mathcal{S}$ a \emph{Y-glued bipartite}
if (a) it is non-acyclic, having an acyclic extension
(as assumed throughout this section),
(b) it has a unique bipartite subposet $\mathcal{A}$ of
Algorithm~\ref{ext.alg},
and (c) the posets $\mathcal{S}_1$ and $\mathcal{S}_2$ constructed by 
Algorithm~\ref{ext.alg} are in Y-class.
\end{definition}

\begin{proposition}\label{y.prop}
Suppose that $\mathcal{S}$ is non-acyclic, and that it has an acyclic extension.
Then monotone equivalence holds for a continuous-time Markov process
if and only if it is a Y-glued bipartite poset.
\end{proposition}

\subsection{The proof of sufficiency}

In order to prove the sufficiency of Definition~\ref{y-glued.bipartite}
for weak monotonicity equivalence
we need the following result of~\cite{fm2001}.

\begin{theorem}\label{y.extension}
{\rm(Section~5 of~\cite{fm2001})}
Let $\mathcal{S}$ be a poset of Y-class.
Suppose that $\mathcal{A}$ is a bipartite poset of Figure~\ref{bipartite}(i)
and that $\hat{\mathcal{A}}$ is its extension of Figure~\ref{bipartite}(ii).
Then a stochastically monotone system $(P_\alpha: \alpha\in\mathcal{A})$
of probability measures on $\mathcal{S}$
can be extended to a stochastically monotone system
$(P_\alpha: \alpha\in \hat{\mathcal{A}})$ with $P_c$ satisfying
$P_{a_i} \preceq P_c \preceq P_{b_j}$
for $i=1,\ldots,m$ and $j=1,\ldots,n$.
\end{theorem}

In this section we assume that $\mathcal{S}$ is Y-glued bipartite.
Thus, we can uniquely identify the poset $\mathcal{A}$
of Definition~\ref{bipartite} and its acyclic extension $\hat{\mathcal{A}}$
consisting all $a_1,\ldots,a_m$ covered by $c$
and $b_1,\ldots,b_n$ covering $c$,
and construct the acyclic extension $\hat{\mathcal{S}}$
of $\mathcal{S}$ by Algorithm~\ref{ext.alg}.

\begin{lemma}\label{y.lemma}
Suppose that $(P_x: x\in\mathcal{A})$ is a stochastically monotone
system of probability measures on $\mathcal{S}$.
Then we can find a probability measure $\hat{P}_c$ on $\hat{\mathcal{S}}$
and some $\theta\in (0,1]$ so that
$\hat{P}_x = \theta P_x + (1-\theta) I_x$, $x\in\mathcal{A}$,
and $\hat{P}_c$ form
a stochastically monotone system
$(\hat{P}_x: x\in\hat{\mathcal{A}})$ of probability measures on $\hat{\mathcal{S}}$.
\end{lemma}

\begin{proof}
Let $\mathcal{S}_1$ and $\mathcal{S}_2$ be the Y-class posets constructed by
Algorithm~\ref{ext.alg}.
Then we can define a stochastically monotone system
$(Q'_x: x\in\mathcal{A})$ of probability measures on $\mathcal{S}_1$ by
setting for each $x \in\mathcal{A}$
\begin{equation*}
Q'_x(y) = \begin{cases}
  P_x(y) & \mbox{ if $y \in S_1\setminus\{c\}$; } \\
  1 - P_x(S_1\setminus\{c\}) & \mbox{ if $y = c$. }
\end{cases}
\end{equation*}
By Theorem~\ref{y.extension}
we can extend it to a stochastically monotone system
$(Q'_x: x\in \hat{\mathcal{A}})$ of probability measures on $\mathcal{S}_1$.
In the exact same way by replacing $\mathcal{S}_1$ with $\mathcal{S}_2$,
we can find a stochastically monotone system
$(Q''_x: x\in\hat{\mathcal{A}})$ of probability measures on $\mathcal{S}_2$.
Let $p_c = Q'_c(S_1\setminus\{c\}) + Q''_c(S_2\setminus\{c\})$,
and set $\theta = 1/p_c$ if $p_c > 1$; otherwise, set $\theta = 1$.
We can introduce
$$
\hat{P}_c(y) = \begin{cases}
  \theta Q'_c(y) & \mbox{ if $y \in S_1\setminus\{c\}$; } \\
  \theta Q''_c(y) & \mbox{ if $y \in S_2\setminus\{c\}$; } \\
  1 - \theta p_c & \mbox{ if $y = c$. }
\end{cases}
$$

Consider an up-set $U$ in $\hat{\mathcal{S}}$.
We will verify that $\hat{P}_{a_i}(U) \le \hat{P}_c(U)$.
Observe that $U \cap S_1$ is an up-set in $\mathcal{S}_1$
and that $U \cap S_2$ is an up-set in $\mathcal{S}_2$.
If $c \not\in U$ then $a_i \not\in U$ and
\begin{align*}
\hat{P}_{a_i}(U)
& = \theta P_{a_i}(U \cap S_1)
+ \theta P_{a_i}(U \cap S_2)
\\ & = \theta Q'_{a_i}(U \cap S_1)
+ \theta Q''_{a_i}(U \cap S_2)
\\ & \le \theta Q'_{c}(U \cap S_1)
+ \theta Q''_{c}(U \cap S_2)
= \hat{P}_c(U)
\end{align*}
If $c \in U$ then
\begin{align*}
\hat{P}_{a_i}(\hat{S}\setminus U)
& \ge \theta P_{a_i}(S_1\setminus U) + \theta P_{a_i}(S_2\setminus U)
\\ & = \theta Q'_{a_i}(S_1\setminus U) + \theta Q''_{a_i}(S_2\setminus U)
\\ & \ge \theta Q'_{c}(S_1\setminus U) + \theta Q''_{c}(S_2\setminus U)
  = \hat{P}_c(\hat{S}\setminus U)
\end{align*}
Similarly we can verify that $\hat{P}_c(U) \le \hat{P}_{b_j}(U)$,
and complete the proof.
\end{proof}

Now consider a stochastically monotone transition probability $P=(P_x:x\in\mathcal{S})$.
By Lemma~\ref{y.lemma}
the subsystem $(P_x:x\in\mathcal{A})$ allows us to produce a stochastically
monotone system $(\hat{P}_x:x\in\hat{\mathcal{A}})$
and to extend it to $\hat{P}=(\hat{P}_x:x\in\hat{\mathcal{S}})$
with $\hat{P}_x = \theta P_x + (1-\theta) I_x$, $x \in\mathcal{S}$,
for some $\theta\in (0,1]$.
Since $\hat{\mathcal{S}}$ is acyclic and $\hat{P}$ is stochastically monotone,
a random monotone map $(\xi_x:x\in\mathcal{S})$
realizes $(P_x:x\in\mathcal{S})$ by Proposition~\ref{fm.claim}.
Hence, $P$ is realizably weakly monotone,
which completes the proof of Proposition~\ref{y.prop}.

\section{A conjecture by Dai Pra et al.}\label{pra.conjecture.section}
\setcounter{equation}{0}\setcounter{figure}{0}

Dai Pra~{et al.}~\cite{pra} introduced
the following subclass of posets,
and gave Proposition~\ref{pra.conjecture}
as a conjecture on monotonicity equivalence.

\begin{definition}\label{w-glued.diamond}
We call $\mathcal{S}$ a \emph{W-glued diamond}
if (i) it has a unique subposet of diamond
with elements $a, b, c, d$ [as in Figure~\ref{named.posets}(i)],
(ii) it has no induced subposet of $\mathcal{S}_1$
and $\hat{\mathcal{S}}_4$ in Figure~\ref{forbidden}
(including the dual of $\mathcal{S}_1$),
and (iii) the Hasse diagram created from $\mathcal{S}$
by removing the arc $(a,b)$, $(a,c)$, $(b,d)$, $(c,d)$ of the diamond
belongs to W-class.
\end{definition}

\begin{figure}[h]
\begin{center}
\begin{tabular}{ccc}

\begin{tikzpicture}
  \node [mynode] (d) at (0,0) {};
  \node [mynode] (b) at (-1,-1) {};
  \node [mynode] (c) at (1,-1) {};
  \node [mynode] (a) at (0,-2) {};
  \node [mynode] (x) at (0,-3) {};
  \draw
  (d) -- (b)
  (d) -- (c)
  (b) -- (a)
  (c) -- (a)
  (a) -- (x);
\end{tikzpicture}

& \hspace{5ex} &

\begin{tikzpicture}
  \node [mynode] (d) at (0,0) {};
  \node [mynode] (b) at (-1,-1) {};
  \node [mynode] (c) at (1,-1) {};
  \node [mynode] (a) at (0,-2) {};
  \node [mynode] (y) at (2,0) {};
  \node [mynode] (x) at (2,-2) {};
  \draw
  (d) -- (b)
  (d) -- (c)
  (b) -- (a)
  (c) -- (a)
  (c) -- (y)
  (c) -- (x);
\end{tikzpicture}

\\
(a) $\mathcal{S}_1$
& & (b) $\hat{\mathcal{S}}_4$
\end{tabular}
\caption{Posets not allowed in W-glued diamond}
\label{forbidden}
\end{center}
\end{figure}
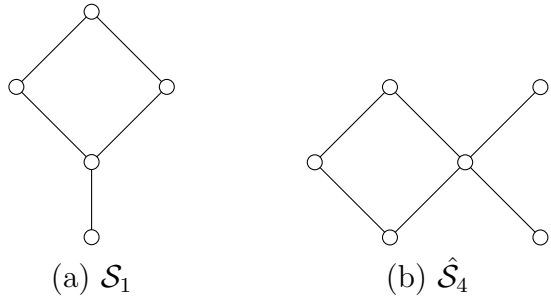

\begin{proposition}\label{pra.conjecture}
Suppose that $\mathcal{S}$ has no acyclic extension.
Monotone equivalence holds for a continuous-time Markov process
if and only if $\mathcal{S}$ is a W-glued diamond.
\end{proposition}

\begin{remark}\label{s4.remark}
In conjunction with examinations by \cite{pra}
we use the symbol $\hat{\mathcal{S}}_4$ in Definition~\ref{w-glued.diamond}(ii)
and indicate it as an extension of $\mathcal{S}_4$ in Proposition~3.1 of~\cite{pra}.
In general if $\mathcal{S}$ has an induced subposet of $\mathcal{S}_1$
or $\hat{\mathcal{S}}_4$
then monotonicity equivalence fails for a continuous-time Markov process
by Proposition~4.4 of~\cite{pra}.
Hence, Definition~\ref{w-glued.diamond}(ii) is necessary for
monotonicity equivalence.
\end{remark}

\subsection{The proof of necessity}
\label{nec.conjecture}

In the rest of this section
we assume that $\mathcal{S}$ has no acyclic extension.
Dai Pra~{et al.}~\cite{pra}
have almost completed the proof of necessity of Proposition~\ref{pra.conjecture}.
The following lemma summarizes their results on
(i) the subdivided bowtie~$\mathcal{S}_4$ with height~$3$
in Proposition~4.4,
(ii) the double-bowtie~$\mathcal{S}_8$
in Proposition~4.5,
and (iii) the $k$-crown with some $k \ge 3$
in Remark~4.6 of~\cite{pra}.

\begin{lemma}\label{pra.lemma}
Monotone equivalence fails for a continuous-time Markov process
if $\mathcal{S}$ has an induced subposet of
(i) subdivided bowtie with height at least~$3$,
(ii) double-bowtie, or
(iii) $k$-crown with some $k \ge 3$.
\end{lemma}

In Proposition~5.11 of~\cite{fm2001}
we found that if $\mathcal{S}$ has no acyclic extension
then it has an induced subposet of Lemma~\ref{pra.lemma}(i)--(iii) or diamond.
Thus, in order to examine monotonicity equivalence
we can further assume that
$\mathcal{S}$ has an induced subposet of diamond.
The next lemma indicates that
such cyclic subposet must be necessarily unique
as required in Definition~\ref{w-glued.diamond}(i).

\begin{lemma}\label{unique.diamond}
Suppose that $\mathcal{S}$ has
an induced cyclic subposet $\mathcal{A}$ of Figure~\ref{named.posets}(i).
Weak monotonicity equivalence fails for $(\mathcal{S},\mathcal{A})$ if
$\mathcal{S}$ has an induced cyclic subposet $\mathcal{S}'$ of
(i) bowtie, or
(ii) diamond distinct from $\mathcal{A}$.
\end{lemma}

\begin{proof}
In either of the cases,
monotonicity equivalence fails for
a system $(P_x: x\in\mathcal{A})$ of probability measures on $\mathcal{S}$;
see (i) Example~4.8 and (ii) Example~1.1 of~\cite{fm2001}.
And it is easily verified that weak monotonicity equivalence also fails
if $A \cap S' = \varnothing$.
In the following we will show that
it also fails when $A \cap S' \neq \varnothing$.

(i)
Let $\mathcal{S}'$ be the bowtie
labeled as in Figure~\ref{named.posets}(ii).
Without loss of generality we assume that $e \in A\cap S'$.
Suppose that either $e = a$ or $e = d$.
Since $e$ and $f$ are incomparable,
we can see that $f \not\in A$. Furthermore, we can
assume that $g \neq b$ and $h \neq c$
(which does not exclude the case that $g=c$ or $h=b$).
Then we can construct the stochastically monotone system
\begin{equation*}
P_a := \frac{1}{2} I_{\{e,f\}}; \quad
P_b := \frac{1}{2} I_{\{f,h\}}; \quad
P_c := \frac{1}{2} I_{\{f,g\}}; \quad
P_d := \frac{1}{2} I_{\{g,h\}}.
\end{equation*}
Let $\xi_x$ be a random monotone map from $\mathcal{A}$ to $\mathcal{S}$.
If $\xi_x$ realizes
$(\theta P_x + (1-\theta)I_x : x\in\mathcal{A})$
for some $\theta \in (0,1]$,
then we must be able to observe that
$\xi_d=g$ implies $\xi_b=f$ and $\xi_d=h$ implies $\xi_c=f$,
and therefore, that
\begin{equation*}
  \mathbb{P}(\xi_a = f) \ge
  \mathbb{P}(\xi_d = g,\, \xi_a = f)
+ \mathbb{P}(\xi_d = h,\, \xi_a = f)
\ge \theta,
\end{equation*}
which contradicts that $\theta P_a(f) + (1-\theta)I_a(f) = \theta/2$;
thus, $(P_x: x\in\mathcal{A})$ cannot be realizably weakly monotone.

Suppose that $e = b$ or $e = c$.
Then $c \not\in \{g,h\}$.
Without loss of generality
we can assume that $e = b$ and $g \not\in A$.
Let
\begin{equation*}
P_a := \frac{1}{2} I_{\{e,f\}}; \quad
P_b := \frac{1}{2} I_{\{e,g\}}; \quad
P_c := \frac{1}{2} I_{\{f,g\}}; \quad
P_d := \frac{1}{2} I_{\{g,h\}}.
\end{equation*}
be the stochastically monotone system
of probability measures on $S$.
If $\xi_x$ realizes $(\theta P_x + (1-\theta)I_x : x\in\mathcal{A})$
then we have
\begin{equation*}
  \mathbb{P}(\xi_d = g) \ge
  \mathbb{P}(\xi_a = e,\, \xi_d = g)
+ \mathbb{P}(\xi_a = f,\, \xi_d = g)
\ge \theta,
\end{equation*}
which contradicts that $\theta P_d(g) + (1-\theta)I_d(g) = \theta/2$.

(ii)
Let $\mathcal{S}'$ be
the second induced subposet of diamond labeled as $a'$, $b'$, $c'$, and $d'$
in the scheme similar to Figure~\ref{named.posets}(i).
If $b, c \in S'$,
then $S$ must have an induced subposet of $\mathcal{S}_5$
in Proposition~3.1 of~\cite{pra},
and therefore, monotonicity equivalence fails by
Proposition~4.4 of~\cite{pra}.
Thus, we can assume without loss of generality
that $b \not\in S'$,  and $c \neq b'$.
Then we can construct the stochastically monotone system
\begin{equation*}
P_a := \frac{1}{2} I_{\{a',b'\}}; \quad
P_b := \frac{1}{2} I_{\{b',c'\}}; \quad
P_c := \frac{1}{2} I_{\{a',d'\}}; \quad
P_d := \frac{1}{2} I_{\{b',d'\}}.
\end{equation*}
If $\xi_x$ realizes $(\theta P_x + (1-\theta)I_x : x\in\mathcal{A})$,
then we have
\begin{equation*}
  \mathbb{P}(\xi_b = b') \ge
  \mathbb{P}(\xi_a = b',\, \xi_b = b',\, \xi_d = d')
+ \mathbb{P}(\xi_d = b',\, \xi_b = b',\, \xi_a = a')
\ge \theta,
\end{equation*}
which contradicts that $\theta P_b(b') + (1-\theta)I_b(b') = \theta/2$.
\end{proof}

The Hasse diagram of Definition~\ref{w-glued.diamond}(iii)
consists of four disconnected components,
each of which contains $a$, $b$, $c$, or $d$.
The next lemma justifies the necessity of W-class for each component
to attain monotonicity equivalence.

\begin{lemma}\label{no.yposet}
Suppose that the diamond $\mathcal{A}$ of Lemma~\ref{unique.diamond} is a
unique induced cyclic subposet of $\mathcal{S}$.
Then weak monotonicity equivalence fails for $(\mathcal{S},\mathcal{A})$ if
$\mathcal{S}$ has an induced subposet $\mathcal{S}'$ of
Y-poset which shares at most one element with $\mathcal{A}$.
\end{lemma}

\begin{proof}
Let $\mathcal{S}'$ be the Y-poset
labeled as in Figure~\ref{named.posets}(iii).
Observe that if $f \in A$
then $\mathcal{S}$ must have an induced subposet
of $\mathcal{S}_1$ or $\hat{\mathcal{S}}_4$,
and therefore, weak monotonicity equivalence fails;
see Remark~\ref{s4.remark}.
Thus, we assume that $f \not\in A$,
and that $h \not\in A$ and $c \not\in S'$ without loss of generality.
Construct the stochastically monotone system
\begin{equation*}
P_a := \frac{1}{2} I_{\{e,f\}}; \quad
P_b := \frac{1}{2} I_{\{e,g\}}; \quad
P_c := \frac{1}{2} I_{\{e,h\}}; \quad
P_d := \frac{1}{2} I_{\{g,h\}}.
\end{equation*}
Suppose that $\xi_x$ realizes $(\theta P_x + (1-\theta)I_x : x\in\mathcal{A})$.
Then
$\mathbb{P}(\xi_a = f,\, \xi_b = g,\, \xi_c = h) \ge \theta/2$
implies that $\mathcal{S}$ has the induced subposet of diamond
on $\{f,g,h,d\}$,
which contradicts the assumption of unique cyclic subposet in $\mathcal{S}$.
Hence, it cannot be realizably weakly monotone.
\end{proof}

A stochastically monotone
$(P_x: x\in\mathcal{A})$ can be extended to a system $(P_x: x\in\mathcal{S})$
by setting $P_x = P_d$ if either $b < x$ or $c < x$;
$P_x = P_a$ if neither $b \ge x$ nor $c \ge x$.
Then the resulting system
$(P_x: x\in\mathcal{S})$ of probability measures on $\mathcal{S}$
is stochastically monotone but not realizably weakly monotone
in each case presented in the proof
of Lemma~\ref{unique.diamond} and~\ref{no.yposet}.
Thus, together with Remark~\ref{s4.remark}
we have shown that it is necessary for $\mathcal{S}$ to be W-glued diamond
in Proposition~\ref{pra.conjecture}.

\subsection{Recursive inverse transforms for W-class}\label{rit}

In this subsection we discuss a construction of random monotone map
in preparation for the sufficiency proof of Proposition~\ref{pra.conjecture}.
We refer the readers to Section~2 of~\cite{classw}
for the proof of Proposition~\ref{map.prop} and further investigation
of recursive inverse transforms.

Let $W$ be a tree (i.e., an acyclic and connected graph),
and let $\tau$ be a maximal or a minimal element of $\mathcal{W}$.
We can introduce a partial ordering for vertices $x,y\in W$,
denoted by $x\le_{\tau}y$, if the path from $\tau$ to $x$ on $W$
contains the path from $\tau$ to $y$.
We call the top element $\tau$ the \emph{root},
and the poset the \emph{rooted tree}, denoted by $(W,\tau)$.
In Algorithm~\ref{rp.tree}
we construct a collection of rooted subtrees
$(W^{(\kappa)},u_1^{(\kappa)})$, $\kappa\in K$,
indexed by a rooted plane tree $(K,1)$.

\begin{algorithm}\label{rp.tree}
We set $(W^{(1)},u_1^{(1)}) = (W,\tau)$ as the rooted tree at the root index
$1$ and start with $K=\{1\}$.
Provided the rooted subtree $(W^{(\kappa)},u_1^{(\kappa)})$ at a leaf
$\kappa\in K$, we can pursue children of it recursively.
\begin{enumerate}
\renewcommand{\labelenumi}{(\roman{enumi})}
\item
In the first step we identify the unique path
\begin{equation}\label{u.path}
  \mathbf{u}^{(\kappa)} = (u_1^{(\kappa)},u_2^{(\kappa)},\ldots,u_*^{(\kappa)})
\end{equation}
until it hits the vertex $u_*^{(1)}$ with no successor or more than one
successor in the rooted tree $(W^{(\kappa)},u_1^{(\kappa)})$.
If $u_*^{(1)}$ has no successor, there is no children and $\kappa\in K$
is labeled as ``terminated,''
indicated by $C(\kappa)=\varnothing$.
Otherwise, we can find the collection
$\{u_1^{(\sigma_1)},\ldots,u_1^{(\sigma_M)}\}$
of successors
[that is, the collection of vertices that $u_*^{(\kappa)}$ covers
in the rooted tree $(S,\tau)$]
indexed by
\begin{equation}\label{plane.order}
  C(\kappa) = \{\sigma_1,\ldots,\sigma_M\},
\end{equation}
and enumerate them
from the first $\sigma_1$ to the last $\sigma_M$;
thus, introducing a linear ordering on $C(\kappa)$.
\item
In the second step
we extend $K$ to $K\cup C(\kappa)$,
where each successor $u_1^{(\sigma_i)}$ is identified with a leaf of a subtree component
$W^{(\sigma_i)}$ of the forest graph obtained from 
$W^{(\kappa)}$ by removing vertices in \eqref{u.path},
and $(W^{(\sigma_i)},u_1^{(\sigma_i)})$, $i=1,\ldots,M$,
become children of $(W^{(\kappa)},u_1^{(\kappa)})$.
\end{enumerate}
We identify all the rooted subtrees
by repeating these steps recursively until
all the terminal indices are labeled as ``terminated.''
\end{algorithm}

We now consider a poset $\mathcal{W}$ of Class W,
and identify the Hasse diagram $W$ with rooted tree $(W,\tau)$.
For each $x\in W$ we can define
the closed section $(\leftarrow,x] = \{z\in W: z\le_{\tau}x\}$
and the open section $(\leftarrow,x) = \{z\in W: z<_{\tau}x\}$
by the rooted tree $(W,\tau)$.
We can observe that each of these sections is either a down-set or an up-set
of $\mathcal{W}$.
For a measure $P$ (i.e., a nonnegative additive set function) on $W$
we can introduce distribution functions $F(x)$ and $F(x-)$ on $W$ respectively by
\begin{equation*}
  F(x) = P((\leftarrow,x])
    \mbox{ and }
  F(x-) = P((\leftarrow,x)) = F(x) - P(\{x\})
\end{equation*}
for $x\in W$.
$F$ is stochastically smaller than $F'$ on $\mathcal{W}$, denoted by $F\preceq F'$,
if
\begin{equation}\label{pdf.stocle}
  \begin{cases}
    F(x) \ge F'(x)
    & \mbox{if $(\leftarrow,x]$ is a down-set of $\mathcal{W}$; } \\
    F(x) \le F'(x)
    & \mbox{if $(\leftarrow,x]$ is an up-set of $\mathcal{W}$ }
  \end{cases}
\end{equation}
for each $x\in W$.
The comparability between $F$ and $F'$ implies a common mass on the whole space,
allowing us to generalizes the stochastic monotonicity for systems of measures
when $\mathcal{W}$ is acyclic.

For the rooted plane tree $(K,1)$ of indices
we can similarly define a distribution function $\mu(\kappa)$ on $K$ if it satisfies
\begin{equation*}
  \mu(\kappa)
  \ge\mu(\kappa-) = \sum_{\sigma_i\in C(\kappa)}\mu(\sigma_i)
  \mbox{ for each $\kappa\in K$, }
\end{equation*}
where the summation above is zero when $C(\kappa)=\varnothing$.
We say that $\mu$ is interlaced with a distribution function $F$ on $W$
if $\mu$ and $F$ satisfy $\mu(1)=F(\tau)$ and
\begin{equation}\label{interlace}
  \mu(\kappa-)\le F(u_*^{(\kappa)})\le F(u_1^{(\kappa)})\le\mu(\kappa)
\end{equation}
for each path $\mathbf{u}^{(\kappa)}$ of \eqref{u.path}.
The following construction of $\mu$ is easily verified.

\begin{lemma}\label{interlace.lem}
Let $F_a,F_d,F$ be distribution functions on $\mathcal{W}$.
If $F_a\preceq F\preceq F_d$ then
\begin{equation}
  \mu(\kappa)=\max\{F_{a}(u_1^{(\kappa)}),F_{d}(u_1^{(\kappa)})\},
  \quad\kappa\in K,
\end{equation}
is interlaced with $F$.
\end{lemma}

In what follows we consider a stochastically monotone system $(F_\alpha:\alpha\in\mathcal{A})$ of
distribution functions on $\mathcal{W}$,
and assume that there exists an interlaced distribution function $\mu$ on $K$
with all $F_\alpha$'s.
We set $\hat{W}^{(1)} = W$, and extend
$W^{(\kappa)}$ to $\hat{W}^{(\kappa)} = W^{(\kappa)}+u_*^{(\hat{\kappa})}$
if $\hat{\kappa}$ is the parent of $\kappa<_{\tau} 1$.
Similarly we extend the path $\mathbf{u}^{(\kappa)}$
by setting
$\hat{\mathbf{u}}^{(1)} =\mathbf{u}^{(1)}$ and
\begin{equation}\label{u.extend}
  \hat{\mathbf{u}}^{(\kappa)} =\mathbf{u}^{(\kappa)}+u_*^{(\hat{\kappa})}
  =(u_*^{(\hat{\kappa})},u_1^{(\kappa)},\ldots,u_*^{(\kappa)})
\end{equation}
to include the tail element
$u_*^{(\hat{\kappa})}$ of the parent path $\mathbf{u}^{(\hat{\kappa})}$
if $\kappa<_{\tau} 1$.
Then we can introduce a system $F^{(\kappa)}=(F_\alpha^{(\kappa)}:\alpha\in\mathcal{A})$
distribution functions on $\hat{W}^{(\kappa)}$ by setting
$F_\alpha^{(1)}=F_\alpha$ at $\kappa=1$ and
\begin{equation}\label{f.recursion}
  F_\alpha^{(\kappa)}(x) = \begin{cases}
    F_\alpha(x) & \mbox{if $x\in W^{(\kappa)}$;} \\
    \mu(\kappa) & \mbox{if $x=u_*^{(\hat{\kappa})}$}
  \end{cases}
\end{equation}
for $\kappa<_{\tau} 1$.
For each vertex $\sigma_j$ of the linearly ordered set
$C(\kappa) = \{\sigma_1,\ldots,\sigma_M\}$ of \eqref{plane.order}
we can introduce $\mu^{(\kappa)}\lceil\sigma_j\rceil$ and
$\mu^{(\kappa)}\lfloor\sigma_j\rfloor$ by
\begin{align}\label{lex.ceil}
  \mu^{(\kappa)}\lceil\sigma_j\rceil &= \sum_{i=1}^j \mu(\sigma_i);
  \\ \label{lex.floor}
  \mu^{(\kappa)}\lfloor\sigma_j\rfloor &= \sum_{i=1}^{j-1} \mu(\sigma_i)
\end{align}

Algorithm~\ref{map.alg} recursively constructs an inverse transform
$X_{\mu,\alpha}^{(\kappa)}$
from $[0,\mu(\kappa))$ to $\hat{W}^{(\kappa)}$,
allowing us to obtain Proposition~\ref{map.prop}.

\begin{algorithm}\label{map.alg}
If $C(\kappa)=\varnothing$ then $\mu(\kappa-)=0$.
Otherwise, we have the linearly ordered set $C(\kappa)$ of \eqref{plane.order}
for which the inverse transform $X_{\mu,\alpha}^{(\sigma_i)}$
from $[0,\mu(\sigma_i))$ to $\hat{W}^{(\sigma_i)}$ can be constructed by recursion.
We set for $\omega\in [0,\mu(\kappa-))$
\begin{equation}\label{map.recursion}
  X_{\mu,\alpha}^{(\kappa)}(\omega) = X_{\mu,\alpha}^{(\sigma_i)}(\omega-\mu^{(\kappa)}\lfloor\sigma_i\rfloor)
  \mbox{ if $\mu^{(\kappa)}\lfloor\sigma_i\rfloor\le\omega<\mu^{(\kappa)}\lceil\sigma_i\rceil$}
\end{equation}
for some $\sigma_i\in C(\kappa)$.
We extend the map $X_{\mu,\alpha}^{(\kappa)}$ from $[0,\mu(\kappa))$
to $\hat{W}^{(\kappa)}$ by setting for $\omega\in [\mu(\kappa-),\mu(\kappa))$,
\begin{equation}\label{map.inverse}
  X_{\mu,\alpha}^{(\kappa)}(\omega)
  = \bigwedge\nolimits_{\tau}\{x\in\hat{\mathbf{u}}^{(\kappa)}: \omega<F_\alpha^{(\kappa)}(x)\}
\end{equation}
which returns the minimum in the linearly ordered set
$(\hat{\mathbf{u}}^{(\kappa)},\le_{\tau})$ of \eqref{u.extend}.
\end{algorithm}

In the construction of $X_{\mu,\alpha}^{(\kappa)}$ in Algorithm~\ref{map.alg}
it should be noted that $X_{\mu,\alpha}^{(\kappa)}(\omega) = u_*^{(\kappa)}$
if and only if
\begin{equation}\label{map.i}
  \omega\in
  I_{\mu,\alpha}^{(\kappa)} =
  \left[\mu(\kappa-),F_\alpha(u_*^{(\kappa)})\right)
    \cup
    \bigcup_{\sigma_i\in C(\kappa)}\left\{
    \mu^{(\kappa)}\lfloor\sigma_i\rfloor +
       \left[F_\alpha(u_1^{(\sigma_i)}),\mu(\sigma_i)\right)
         \right\}
\end{equation}
where the union over $C(\kappa)$ is the empty set if $C(\kappa)=\varnothing$.

\begin{proposition}\label{map.prop}
The recursive inverse transform $X_{\mu,\alpha}^{(\kappa)}$ of Algorithm~\ref{map.alg} realizes
$F_\alpha^{(\kappa)}$ in the sense that
\begin{equation}\label{map.realize}
  \lambda\left(\left\{\omega:
  X_{\mu,\alpha}^{(\kappa)}(\omega)\in(\leftarrow,x]
  \right\}\right)
  = F_\alpha^{(\kappa)}(x),
  \quad x\in\hat{W}^{(\kappa)},
\end{equation}
where $\lambda$ is the Lebesgue measure on $[0,\mu(\kappa))$.
Furthermore, the system $(X_{\mu,\alpha}^{(\kappa)}:\alpha\in\mathcal{A})$
is monotone so that
\begin{equation}\label{map.monotone}
  X_{\mu,\alpha}^{(\kappa)}(\omega)\le X_{\mu,\beta}^{(\kappa)}(\omega)
  \mbox{ in $\hat{\mathcal{W}}^{(\kappa)}$ for any $\omega\in [0,\mu(\kappa))$}
\end{equation}
whenever $\alpha\le\beta$ in $\mathcal{A}$.
\end{proposition}

In Proposition~\ref{map.prop}
we simply write $X_{\mu,\alpha}$ for the recursive inverse transform
$X_{\mu,\alpha}^{(1)}$ from $[0,\mu(1))$ to $\mathcal{W}$.
In the light of Proposition~\ref{map.prop}
we can generalize the realizable monotonicity for a system of measures
on an acyclic poset
if there exists a system $(X_{\mu,\alpha}:\alpha\in\mathcal{A})$ of maps
satisfying \eqref{map.realize}--\eqref{map.monotone} for $\mathcal{W}=\hat{\mathcal{W}}^{(1)}$
with $\kappa=1$.
Similarly we call $P$ (or $F$) realizably weakly monotone if a system
$(P + \theta I)$ of measures is realizably monotone for some $\theta\ge 0$.
The following result is an immediate corollary to Proposition~\ref{map.prop}.

\begin{corollary}\label{w.cor}
Suppose that a system $F=(F_\alpha:\alpha\in\mathcal{A})$ of distribution functions
is stochastically monotone on a poset $\mathcal{W}$ of W-class,
and that $\mathcal{A}$ has the minimum element $a$ and the maximum
element $d$.
Then the distribution $\mu$ of Lemma~\ref{interlace.lem} is interlaced with $F$,
and $(X_{\mu,\alpha}:\alpha\in\mathcal{A})$
is monotone and realizes $F$.
\end{corollary}

\subsection{The proof of sufficiency}
\label{w.sec}

We continue the setting of Section~\ref{nec.conjecture},
and examine a monotonicity equivalence problem
for the sufficiency proof of Proposition~\ref{pra.conjecture}
when $\mathcal{S}$ is a W-glued diamond.
We start from the diamond $\mathcal{A}$ of Figure~\ref{named.posets}(i),
and glue it with four posets of W-class, say $\mathcal{W}_a$, $\mathcal{W}_b$,
$\mathcal{W}_c$, and $\mathcal{W}_d$,
and construct a W-glued diamond $\mathcal{S}$ in the following manner:
The diamond $\mathcal{A}$ is glued
(i) at $a$ with a minimal element of $\mathcal{W}_a$,
(ii) at $b$ with a maximal or minimal element of $\mathcal{W}_b$,
(iii) at $c$ with a maximal or minimal element of $\mathcal{W}_c$,
and (iv) at $d$ with a maximal element of $\mathcal{W}_d$.

In Lemma~\ref{y.me}
we consider the subposet $\mathcal{A}'$ of the diamond
$\mathcal{A}$ induced on $A'=\{a,b,d\}$,
and obtain a poset $\mathcal{Y}'$ of Y-class by gluing
$\mathcal{A}'$ at $b$ with a minimal element of
a poset $\mathcal{W}_b$ of W-class.

\begin{lemma}\label{y.me}
Let $\tilde{Q}=(\tilde{Q}_\alpha:\alpha\in\mathcal{A})$ be a system
of measures on $\mathcal{Y}'$.
Let $Q_\alpha=\tilde{Q}_\alpha$, $\alpha=b,d$, be supported by
$\hat{W}^{(\eta_b)}=W_b\cup\{d\}$,
and let
\begin{equation}\label{q.cons}
  Q_\alpha(y)=\begin{cases}
  \tilde{Q}_\alpha(a)+\tilde{Q}_\alpha(b) & \mbox{if $y=b$;}\\
  \tilde{Q}_\alpha(y) & \mbox{if $y\in \hat{W}^{(\eta_b)}\setminus\{b\}$}
  \end{cases}
\end{equation}
for $\alpha=a,c$
so that $Q=(Q_\alpha:\alpha\in\mathcal{A})$
becomes a system of measures on
the subposet $\hat{\mathcal{W}}^{(\eta_b)}$ induced by $\hat{W}^{(\eta_b)}$.
Suppose that $\tilde{Q}$ and  $Q$ are stochastically monotone
respectively on $\mathcal{Y}'$
and $\hat{\mathcal{W}}^{(\eta_b)}$.
Then there exists a modified system of recursive inverse transforms
which is monotone in $\mathcal{Y}'$ and realizes $\tilde{Q}$.
\end{lemma}

\begin{proof}
By Algorithm~\ref{rp.tree} we can start from the rooted tree $(\hat{W}^{(\eta_b)},d)$,
and recursively construct $(W^{(\kappa)},u_1^{(\kappa)})$, $\kappa\in K_b$,
along with a rooted plane tree $(K_b,\eta_b)$ of indices.
Then a distribution $\mu$ of Lemma~\ref{interlace.lem} is interlaced with
the system $F^{(\eta_b)}=(F^{(\eta_b)}_\alpha:\alpha\in\mathcal{A})$
of distribution functions corresponding to $Q$,
and by Algorithm~\ref{map.alg}
we can introduce a monotone system $(X_{\mu,\alpha}^{(\eta_b)}:\alpha\in\mathcal{A})$
of recursive inverse transforms from $[0,\mu(\eta_b))$ to $\hat{\mathcal{W}}^{(\eta_b)}$ which
realizes $F^{(\eta_b)}$.
Then we modify it and obtain $\tilde{X}_{\mu,\alpha}^{(\eta_b)}$ mapping to
$\mathcal{Y}'$ as follows:
First we choose
\begin{equation*}
  \gamma_a = \begin{cases}
    \max\{\tilde{Q}_a(a)-F^{(\eta_b)}_c(b)+\mu(\eta_b-),0\} \\
    & \hspace{-12ex}\mbox{if $\tilde{Q}_a(a)\le F^{(\eta_b)}_c(b)-F^{(\eta_b)}_c(b-)$;}\\
    \mu(\eta_b-)-F^{(\eta_b)}_a(b-)
    & \hspace{-12ex}\mbox{otherwise,}
  \end{cases}
\end{equation*}
so that the interval
\begin{equation}\label{a.int}
  \mu(\eta_b-)+[-\gamma_a,\tilde{Q}_a(a)-\gamma_a)
\end{equation}
is contained in $[F^{(\eta_b)}_c(b-),F^{(\eta_b)}_c(b))$
if $\tilde{Q}_a(a)\le F^{(\eta_b)}_c(b)-F^{(\eta_b)}_c(b-)$;
otherwise, \eqref{a.int} contains $[F^{(\eta_b)}_c(b-),F^{(\eta_b)}_c(b))$.
Then set $\gamma_b = \gamma_d = 0$ and
\begin{equation*}
  \gamma_c = \max\{\tilde{Q}_c(a)-F^{(\eta_b)}_c(b)+\mu(\eta_b-),0\}
\end{equation*}
in such a way that
\eqref{a.int} contains the interval
\begin{equation*}
  \mu(\eta_b-)+[-\gamma_c,\tilde{Q}_c(a)-\gamma_c).
\end{equation*}
For each $\sigma_i\in C(\eta_b)$
we define
$\tilde{F}_b^{(\sigma_i)}(b) = \tilde{F}_d^{(\sigma_i)}(b) = \mu(\sigma_i)$
and
\begin{align*}
  \tilde{F}_a^{(\sigma_i)}(b)
  &= \mu(\sigma_i) -
  \frac{\gamma_a[\mu(\sigma_i)-F^{(\eta_b)}_a(u_1^{(\sigma_i)})]}{\mu(\eta_b-)-F^{(\eta_b)}_a(b-)}
  \\
  \tilde{F}_c^{(\sigma_i)}(b)
  &= \tilde{F}_a^{(\sigma_i)}(b)
  + \dfrac{(\gamma_a-\gamma_c)[\mu(\sigma_i)-\tilde{F}_a^{(\sigma_i)}(b)]}{\gamma_a} .
\end{align*}
so that
\begin{equation*}
  \sum_{\sigma_i\in C(\eta_b)}[\mu(\sigma_i)-\tilde{F}_\alpha^{(\sigma_i)}(b)]
  = \gamma_\alpha
  \mbox{ for $\alpha\in\mathcal{A}$.}
\end{equation*}
For each $\alpha\in\mathcal{A}$
we set
\begin{equation*}
  \tilde{F}_\alpha^{(\eta_b)}(a)=\mu(\eta_b-)+\tilde{Q}_{\alpha}(a)-\gamma_\alpha ,
\end{equation*}
and modify $X_{\mu,\alpha}^{(\eta_b)}$
by replacing $X_{\mu,\alpha}^{(\eta_b)}(\omega) = b$ with
$\tilde{X}_{\mu,\alpha}^{(\eta_b)}(\omega) = a$ if
\begin{equation}\label{i.mod}
  \omega\in\tilde{I}_{\mu,\alpha}^{(\eta_b)}
  = \left[\mu(\eta_b-),\tilde{F}_\alpha^{(\eta_b)}(a)\right)
  \cup
  \bigcup_{\sigma_i\in C(\eta_b)}\left\{
  \mu^{(\eta_b)}\lfloor\sigma_i\rfloor +
     \left[\tilde{F}_\alpha^{(\sigma_i)}(b),\mu(\sigma_i)\right)
       \right\}.
\end{equation}
Since
$\tilde{I}_{\mu,b}^{(\eta_b)}=\tilde{I}_{\mu,d}^{(\eta_b)}=\varnothing$,
we only change $X_{\mu,a}^{(\eta_b)}$ and $X_{\mu,c}^{(\eta_b)}$.
The modified system $(\tilde{X}_{\mu,\alpha}^{(\eta_b)}:\alpha\in\mathcal{A})$
clearly realizes $\tilde{Q}$.
Since
\begin{math}
  \tilde{I}_{\mu,c}^{(\eta_b)}\subseteq\tilde{I}_{\mu,a}^{(\eta_b)},
\end{math}
it satisfies
\begin{equation*}
  a=\tilde{X}_{\mu,a}^{(\eta_b)}(\omega)=\tilde{X}_{\mu,c}^{(\eta_b)}(\omega)
  < X_{\mu,a}^{(\eta_b)}(\omega)=b
  \le\tilde{X}_{\mu,b}^{(\eta_b)}(\omega)\le\tilde{X}_{\mu,d}^{(\eta_b)}(\omega)
  \mbox{ if $\omega\in\tilde{I}_{\mu,c}^{(\eta_b)}$;}
\end{equation*}
otherwise,
\begin{equation*}
  \tilde{X}_{\mu,a}^{(\eta_b)}(\omega)
  \le X_{\mu,a}^{(\eta_b)}(\omega)
  \le\tilde{X}_{\mu,\alpha}^{(\eta_b)}(\omega)
  \le\tilde{X}_{\mu,d}^{(\eta_b)}(\omega),
  \quad
  \alpha=b,c;
\end{equation*}
thus, it is monotone.
\end{proof}

In Lemma~\ref{w.lemma}
we consider a W-glued diamond $\tilde{\mathcal{S}}$ consisting of the diamond
$\mathcal{A}$ glued at $b$ with a minimal element of $\mathcal{W}_b$
and at $c$ with a minimal element of $\mathcal{W}_c$.

\begin{lemma}\label{w.lemma}
Let $P=(P_\alpha: \alpha\in\mathcal{A})$ be a stochastically monotone system of
measures on $\tilde{\mathcal{S}}$.
If $P_a(d) = P_d(a) = 0$ then it is realizably weakly monotone.
\end{lemma}

\begin{proof}
Let $p$ be the common mass of $P$ [i.e., $p=P_\alpha(\tilde{S})$ for all $\alpha\in\mathcal{A}$],
and let $\mathcal{Y}'$ and $\mathcal{Y}''$ be the subposet of Y-class
induced respectively on $W_b\cup\{a,d\}$ and $W_c\cup\{a,d\}$.
Then we can construct a stochastically monotone system
$\tilde{Q}'=(\tilde{Q}'_\alpha: \alpha\in\mathcal{A})$ of measures on $\mathcal{Y}'$
with common mass $\tilde{q}'=P_c(Y')$
by setting
\begin{align*}
\tilde{Q}'_a(y) & = \begin{cases}
  \tilde{q}' - P_a(W_b) & \mbox{ if $y = a$; } \\
  P_a(b) & \mbox{ if $y\in W_b$; } \\
  0 & \mbox{ if $y = d$, }
\end{cases} \\
\tilde{Q}'_b(y) & = \begin{cases}
  \tilde{q}' - P_b(W_b\setminus\{b\}) & \mbox{ if $y = b$; } \\
  P_b(y) & \mbox{ if $y\in W_b\setminus\{b\}$; } \\
  0 & \mbox{ if $y = a$ or $d$, } \\
\end{cases} \\
\tilde{Q}'_c(y) & = P_c(y), \\
\tilde{Q}'_d(y) & = \begin{cases}
  0 & \mbox{ if $y = a$; } \\
  P_d(b) & \mbox{ if $y\in W_b$; } \\
  \tilde{q}' - P_d(W_b) & \mbox{ if $y = d$. } \\
\end{cases}
\end{align*}
By interchanging $b$ with $c$ in the above construction,
we obtain the stochastically monotone system
$\tilde{Q}''=(\tilde{Q}''_\alpha: \alpha\in\mathcal{A})$ of measures on $\mathcal{Y}''$
for which $\tilde{q}''=P_b(Y'')$ is the common mass of $\tilde{Q}''$.
Let $I = (I_\alpha:\alpha\in\mathcal{A})$ be the system of unit masses.
By viewing $\tilde{Q}=\tilde{Q}'+\tilde{Q}''$ as a system of measures on $\tilde{\mathcal{S}}$
we obtain
\begin{equation*}
  \tilde{Q} + \theta I = P + \theta^* I
\end{equation*}
with $\theta = \max\{p-\tilde{q},0\}$
and $\theta^* = \max\{\tilde{q}-p,0\}$,
where $\tilde{q}=\tilde{q}'+\tilde{q}''$ is the common mass of $\tilde{Q}$.

Let $\tilde{\mathcal{W}}$ be the poset of W-class
induced on $\tilde{S}\setminus\{a\}$,
and let $\hat{\mathcal{W}}^{(\eta_b)}$ and $\hat{\mathcal{W}}^{(\eta_c)}$
be the subposets respectively induced on $W_b\cup\{d\}$ and $W_c\cup\{d\}$ and rooted at $d$.
Then we can introduce
the stochastically monotone systems $F^{(\eta_b)}$ and $F^{(\eta_c)}$ of distribution functions
respectively on $\hat{\mathcal{W}}^{(\eta_b)}$ and $\hat{\mathcal{W}}^{(\eta_c)}$
by modifying the respective systems $\tilde{Q}'$ and $\tilde{Q}''$ similarly to Lemma~\ref{y.me}.
We combine them by
\begin{equation}\label{f.on.w}
  F_\alpha(y) = \begin{cases}
    F_\alpha^{(\eta_b)}(y) & \mbox{if $y\in W_b$;} \\
    F_\alpha^{(\eta_c)}(y) & \mbox{if $y\in W_c$;} \\
    \tilde{q}+\theta & \mbox{if $y=d$,}
  \end{cases}
\end{equation}
and create a stochastically monotone system $F$ on $\tilde{\mathcal{W}}$.
Along with the interlacing distribution $\mu$ of Lemma~\ref{interlace.lem},
we can construct the recursive inverse transform
\begin{equation*}
  X_{\mu,\alpha}(\omega) = \begin{cases}
    X_{\mu,\alpha}^{(\eta_b)}(\omega)
    & \mbox{if $0\le\omega<\mu^{(1)}\lceil\eta_b\rceil$;}\\
    X_{\mu,\alpha}^{(\eta_c)}(\omega-\mu^{(1)}\lfloor\eta_c\rfloor)
    & \mbox{if $\mu^{(1)}\lfloor\eta_c\rfloor\le\omega<\mu^{(1)}\lceil\eta_c\rceil$;}\\
    d
    & \mbox{if $\mu(1-)\le\omega<\mu(1)$,}
  \end{cases}
\end{equation*}
where $C(1)=\{\eta_b,\eta_c\}$ is linearly ordered.
Here we find
$\mu^{(1)}\lceil\eta_b\rceil=\mu^{(1)}\lfloor\eta_c\rfloor=\mu(\eta_b)=q'$,
$\mu^{(1)}\lceil\eta_c\rceil=\mu(1-)=\tilde{q}$,
and
$\mu(1)=\tilde{q}+\theta$.
Exactly in the same manner to the proof of Lemma~\ref{y.me}
we can introduce the interval $I_{\mu,\alpha}^{(\eta_b)}$ by \eqref{i.mod}
and $I_{\mu,\alpha}^{(\eta_c)}$ by replacing $b$ with $c$,
and modify $X_{\mu,\alpha}^{(\eta_b)}$ and $X_{\mu,\alpha}^{(\eta_c)}$
by setting
(i) $\tilde{X}_{\mu,\alpha}^{(\eta_b)}(\omega)=a$ if $\omega\in I_{\mu,\alpha}^{(\eta_b)}$,
and (ii)
$\tilde{X}_{\mu,\alpha}^{(\eta_c)}(\omega)=a$ if $\omega\in I_{\mu,\alpha}^{(\eta_c)}$.
Furthermore, we set $\tilde{X}_{\mu,\alpha}(\omega)=\alpha$
if $\tilde{q}\le\omega<\tilde{q}+\theta$.
The resulting system
$(\tilde{X}_{\mu,\alpha}:\alpha\in\mathcal{A})$ is monotone on $\tilde{\mathcal{S}}$
and realizes $\tilde{Q}+\theta I$.
\end{proof}

The result of Lemma~\ref{w.lemma}
holds true when the diamond $\mathcal{A}$ is
glued at $b$ and $c$
respectively with a minimal or a maximal element of $\mathcal{W}_b$ and $\mathcal{W}_c$.
In Theorem~\ref{w.theorem}
we extend the poset $\tilde{\mathcal{S}}$ of Lemma~\ref{w.lemma}
to a W-glued diamond $\mathcal{S}$
further glued at $a$ with a minimal element of $\mathcal{W}_a$
and at $d$ with a maximal element of $\mathcal{W}_d$.

\begin{theorem}\label{w.theorem}
If $P=(P_\alpha:\alpha\in\mathcal{A})$ is a stochastically monotone system of measures on
$\mathcal{S}$, it is realizably weakly monotone.
\end{theorem}

\begin{proof}
Let $p$ be the common mass of $P$,
and let $\hat{\mathcal{W}}^{(\eta_a)}$ be the subposet of $\mathcal{S}$
induced on $W_a\cup\{d\}$.
Then we define
the poset $\mathcal{W}'$ of $\hat{\mathcal{W}}^{(\eta_a)}$ and
$\mathcal{W}_d$ glued and rooted at $d$ where
$C(1) = \{\eta_a,\sigma_1,\ldots,\sigma_M\}$
indexes the successors $u_1^{(\eta_a)}=a$ and $u_1^{(\sigma_i)}$'s
covered by $u_1^{(1)}=d$.
We introduce a stochastically monotone system
$P'=(P'_\alpha:\alpha\in\mathcal{A})$ of measures on $\mathcal{W}'$ by
\begin{equation*}
  P'_\alpha(y) = \begin{cases}
    P_d(W_a)-P_\alpha(W_a\setminus\{a\}) & \mbox{if $y=a$;} \\
    P_a(W_d)-P_\alpha(W_d\setminus\{d\}) & \mbox{if $y=d$;} \\
    P_\alpha(y) & \mbox{if $y\in W'\setminus\{a,d\}$,}
  \end{cases}
\end{equation*}
where $p'=P_d(W_a)+P_a(W_d)$ is the common mass of $P'$.
By $(F'_\alpha:\alpha\in\mathcal{A})$ we denote the corresponding system of
distribution functions on $\mathcal{W}'$.

The system $P''=P-P'$ is viewed as a stochastically monotone system of
measures on the W-glued diamond $\tilde{\mathcal{S}}$ of Lemma~\ref{w.lemma},
satisfying $P''_a(d) = P''_d(a) = 0$.
By Lemma~\ref{w.lemma} we obtain a realizably monotone system $\tilde{Q}$ which
satisfies $\tilde{Q}+\theta I=P''+\theta^* I$,
with $\theta=\max\{p''-\tilde{q},0\}$
and $\theta^*=\max\{\tilde{q}-p'',0\}$,
where $\tilde{q}$ and $p''$
are the common masses of $\tilde{Q}$ of $P''$ respectively.
Here we start from the system $\tilde{Q}$ on $\tilde{\mathcal{S}}$,
and introduce a stochastically monotone system
$F=(F_\alpha:\alpha\in\mathcal{A})$ of distribution functions
on the poset $\tilde{\mathcal{W}}$ of
$\hat{\mathcal{W}}^{(\eta_b)}$ and $\hat{\mathcal{W}}^{(\eta_c)}$ glued and rooted at $d$.
In the proof of Lemma~\ref{w.lemma}
we have demonstrated that
a modified system of recursive inverse transforms
$\tilde{X}_{\mu,\alpha}$
is monotone on $\tilde{\mathcal{S}}$ and realizes $\tilde{Q}+\theta I$.

Define the poset $\mathcal{W}$ of $\mathcal{W}'$ and $\tilde{\mathcal{W}}$
glued and rooted at $d$, and enumerate the corresponding
$C(1) = \{\eta_b,\eta_c,\eta_a,\zeta_1,\ldots,\zeta_M\}$.
Then we can extend $F$ on $\mathcal{W}$ by setting
\begin{equation*}
  F_\alpha(y) = \begin{cases}
    F_\alpha(y) & \mbox{if $y\in\tilde{W}\setminus\{d\}$;}\\
    F'_\alpha(y) & \mbox{if $y\in W'\setminus\{d\}$;}\\
    p+\theta^* & \mbox{if $y=d$.}
  \end{cases}
\end{equation*}
Using the interlacing distribution $\mu$ of Lemma~\ref{interlace.lem},
by Algorithm~\ref{map.alg} we can construct the recursive inverse transform
\begin{equation}\label{map.on.w}
  X_{\mu,\alpha}(\omega) = \begin{cases}
    X_{\mu,\alpha}^{(\eta_b)}(\omega)
    & \mbox{if $0\le\omega<\mu^{(1)}\lceil\eta_b\rceil$;}\\
    X_{\mu,\alpha}^{(\eta_c)}(\omega-\mu^{(1)}\lfloor\eta_c\rfloor)
    & \mbox{if $\mu^{(1)}\lfloor\eta_c\rfloor\le\omega<\mu^{(1)}\lceil\eta_c\rceil$;}\\
    X_{\mu,\alpha}^{(\eta_a)}(\omega-\mu^{(1)}\lfloor\eta_a\rfloor)
    & \mbox{if $\mu^{(1)}\lfloor\eta_a\rfloor\le\omega<\mu^{(1)}\lceil\eta_a\rceil$;}\\
    X_{\mu,\alpha}^{(\zeta_i)}(\omega-\mu^{(1)}\lfloor\zeta_i\rfloor)
    & \mbox{if
      $\mu^{(1)}\lfloor\zeta_i\rfloor\le\omega<\mu^{(1)}\lceil\zeta_i\rceil$
      for some $i$;}\\
    d
    & \mbox{if $\mu(1-)\le\omega<\mu(1)$,}\\
  \end{cases}
\end{equation}
where
$\mu^{(1)}\lceil\eta_b\rceil=F_a(b)$,
$\mu^{(1)}\lceil\eta_c\rceil=F_a(b)+F_a(c)=\tilde{q}$,
$\mu^{(1)}\lceil\eta_a\rceil=\tilde{q}+F'_d(a)$,
\begin{math}
  \mu^{(1)}\lceil\zeta_M\rceil=\mu(1-)=\tilde{q}+F'_d(a)+F'_a(d-)
  =\tilde{q}+p'-P_a(d),
\end{math}
and
$\mu(1)=\tilde{q}+p'+\theta$.
The modification $\tilde{X}_{\mu,\alpha}$ of $X_{\mu,\alpha}$
can be made by setting
(i) $\tilde{X}_{\mu,\alpha}^{(\eta_b)}(\omega)=a$ for $\omega\in I_{\mu,\alpha}^{(\eta_b)}$,
(ii) $\tilde{X}_{\mu,\alpha}^{(\eta_c)}(\omega)=a$ for $\omega\in I_{\mu,\alpha}^{(\eta_c)}$,
and (iii) $\tilde{X}_{\mu,\alpha}(\omega)=\alpha$ for $\omega\in [\tilde{q}+p',\tilde{q}+p'+\theta)$
in a manner similar to the proof of Lemma~\ref{w.lemma}.
The resulting $\tilde{X}_{\mu,\alpha}$ is monotone on $\mathcal{S}$,
and realizes $P+\theta^* I$.
\end{proof}

We now consider a stochastically monotone transition probability
$P = (P_x: x\in\mathcal{S})$ on a poset $\mathcal{S}$ of W-glued diamond,
and complete the sufficiency proof of Proposition~\ref{pra.conjecture}.
In a series of constructive proofs from Lemma~\ref{y.me}
to Theorem~\ref{w.theorem} we find a way to develop a monotone map
$\tilde{X}_{\mu,\alpha}$ which realizes $(P_\alpha + \theta I_\alpha: \alpha\in\mathcal{A})$
for some $\theta\ge 0$.
We refer to the following result of Theorem~\ref{acyclic} and Lemma~\ref{index.glued},
and show that $P$ is realizably weakly monotone.
Consequently by Proposition~\ref{weak.equivalence}
we establish Proposition~\ref{pra.conjecture}.

\begin{theorem}[Section~4 of~\cite{fm2001}]\label{acyclic}
Suppose that $\mathcal{W}$ is an acyclic poset.
Then a stochastically monotone system $(P_x:x\in\mathcal{W})$ of
measures on an arbitrary poset $\mathcal{S}$ is realizably monotone.
\end{theorem}

\begin{lemma}\label{index.glued}
Let $\mathcal{S}$ be a poset,
and let $\mathcal{V}'$ and $\mathcal{V}''$ be posets
such that $V'\cap V''=\{\alpha\}$.
Let $\mathcal{V}$ be the poset obtained
from $\mathcal{V}'$ and $\mathcal{V}''$ glued at $\alpha$,
and let $P=(P_v:v\in\mathcal{V})$ be a system of measures on $S$.
If $P'=(P_v:v\in\mathcal{V}')$ and $P''=(P_v:v\in\mathcal{V}'')$
are realizably monotone on $\mathcal{S}$ then so is $P$.
\end{lemma}

In the W-glued diamond $\mathcal{S}$
we view $\mathcal{W}_\alpha$ as an acyclic subposet
glued to the diamond $\mathcal{A}$ at each $\alpha\in\mathcal{A}$.
By Theorem~\ref{acyclic}
we find $(P_x + \theta I_x: x \in\mathcal{W}_\alpha)$
realizably monotone for each acyclic poset $\mathcal{W}_\alpha$, $\alpha\in A$;
thus, so is $P + \theta I$ by Lemma~\ref{index.glued},
and equivalently, $P$ is realizably weakly monotone.

We conclude this subsection by presenting the proof of
Lemma~\ref{index.glued},
and discuss a constructive method of realization for Theorem~\ref{acyclic}
in the next subsection.

\begin{proof}[Proof of Lemma~\ref{index.glued}]
In this proof we assume that $P_x$'s are probability measures on $S$.
Then there exist random monotone maps $(\xi'_v:v\in\mathcal{V}')$
and $(\xi''_v:v\in\mathcal{V}'')$
respectively realizing $P'$ and $P''$.
We consider the pmf $Q'$ on $S^{V'}$
generated by $(\xi'_v:v\in V')$,
and the conditional probability mass function
$Q''(\cdot|x)$
of $(\xi''_v:v\in V''\setminus\{\alpha\})$
given $\xi''_\alpha=x$.
We can introduce a pmf $Q$ on $S^V$ by
\begin{equation*}
  Q(h) =
  Q''(\pi_{V''\setminus\{\alpha\}}(h)|h(\alpha))Q'(\pi_{V'}(h)),
  \quad
  h\in S^V,
\end{equation*}
where $\pi_{V'}(h)$ denotes the map $h$ restricted on $V'$.
A system $(\xi_v:v\in\mathcal{V})$ of $S$-valued random variables
generated by $Q$ is monotone and it marginally realizes $P$.
Thus, $P$ is realizably monotone.
\end{proof}

\subsection{A construction of Nachbin-Strassen realization}
\label{strassen}

Strassen~\cite{strassen} has characterized stochastically ordered measures
$P_1\preceq P_2$ on a poset $\mathcal{S}$ by an ordered
realization of maps $X_1$ and $X_2$ from $[0,p)$ to $\mathcal{S}$.
Similarly to Proposition~\ref{map.prop},
the pair $(X_1,X_2)$ is marginally realized in the sense that
\begin{equation}\label{realize}
  \lambda(\{\omega:X_i(\omega)\in B\})=P_i(B),
  \quad
  i=1,2,
\end{equation}
for any subset $B$ of $S$,
and ordered so that
\begin{equation}\label{ordered}
  X_1(\omega)\le X_2(\omega)
  \mbox{ in $\mathcal{S}$ whenever $\omega\in [0,p)$.}
\end{equation}
The existence of such a pair $(X_1,X_2)$ has been known as Nachbin-Strassen
theorem~\cite{kko}, and given in a general setting of closed partial ordering
(cf. \cite{nachbin}).
Theorem~\ref{acyclic} is viewed as a corollary to this characterization
and proven by induction, which is a straightforward application of
Lemma~\ref{index.glued}.

Let $\mathcal{A}$ be the diamond of Figure~\ref{named.posets}(i),
and let $\mathcal{S}$ be the W-glued diamond of Theorem~\ref{w.theorem}.
In this subsection we consider a stochastically ordered pair,
$P_1\preceq P_2$, of measures on $\mathcal{S}$,
and present a construction of Nachbin-Strassen realization
in Proposition~\ref{w.strassen}.

As in the proof of Lemma~\ref{w.lemma} we define
the subposets $\mathcal{Y}'$ and $\mathcal{Y}''$
induced respectively on $\{a,d\}\cup W_b$ and $\{a,d\}\cup W_c$,
and the subposets $\hat{\mathcal{W}}^{(\eta_b)}$
and $\hat{\mathcal{W}}^{(\eta_c)}$
induced respectively on $\{d\}\cup W_b$ and $\{d\}\cup W_c$.
We set $q'=\max\{P_1(W_b),P_2(W_b)\}$
and $q''=\max\{P_1(W_c),P_2(W_c)\}$,
and introduce measures on $\mathcal{Y}'$ by
\begin{align*}
  \tilde{Q}'_1(y) &= \begin{cases}
    q'-P_1(W_b) & \mbox{if $y=a$;}\\
    0 & \mbox{if $y=d$;}\\
    P_1(y) & \mbox{if $y\in W_b$,}
  \end{cases}
  \\
  \tilde{Q}'_2(y) &= \begin{cases}
    0 & \mbox{if $y=a$;}\\
    q'-P_2(W_b) & \mbox{if $y=d$;}\\
    P_2(y) & \mbox{if $y\in W_b$,}
  \end{cases}
\end{align*}
and measures on $\mathcal{Y}''$ by
\begin{align*}
  \tilde{Q}''_1(y) &= \begin{cases}
    q''-P_1(W_c) & \mbox{if $y=a$;}\\
    0 & \mbox{if $y=d$;}\\
    P_1(y) & \mbox{if $y\in W_c$,}
  \end{cases}
  \\
  \tilde{Q}''_2(y) &= \begin{cases}
    0 & \mbox{if $y=a$;}\\
    q''-P_2(W_c) & \mbox{if $y=d$;}\\
    P_2(y) & \mbox{if $y\in W_c$.}
  \end{cases}
\end{align*}
We modify $\tilde{Q}'_1$ and $\tilde{Q}''_1$
respectively by
\begin{align*}
  Q^{(\eta_b)}_1(y) &= \begin{cases}
    q'-P_1(W_b\setminus\{b\}) & \mbox{if $y=b$;}\\
    0 & \mbox{if $y=d$;}\\
    P_1(y) & \mbox{if $y\in W_b\setminus\{b\}$,}
  \end{cases}
  \\
  Q^{(\eta_c)}_1(y) &= \begin{cases}
    q''-P_1(W_c\setminus\{c\}) & \mbox{if $y=c$;}\\
    0 & \mbox{if $y=d$;}\\
    P_1(y) & \mbox{if $y\in W_c\setminus\{c\}$.}
  \end{cases}
\end{align*}
And we set $Q^{(\eta_b)}_2=\tilde{Q}'_2$
and $Q^{(\eta_c)}_2=\tilde{Q}''_2$, and view them
as measures on $\hat{\mathcal{W}}^{(\eta_b)}$
and $\hat{\mathcal{W}}^{(\eta_c)}$ respectively.

\begin{lemma}
$Q^{(\eta_b)}_1\preceq Q^{(\eta_b)}_2$
on $\hat{\mathcal{W}}^{(\eta_b)}$,
and $Q^{(\eta_c)}_1\preceq Q^{(\eta_c)}_2$
on $\hat{\mathcal{W}}^{(\eta_c)}$.
\end{lemma}

\begin{proof}
Let $F^{(\eta_\alpha)}_i$ be the distribution function corresponding to
$Q^{(\eta_\alpha)}_i$ for each $\alpha=b,c$ and $i=1,2$ introduced by the
rooted tree $(\hat{\mathcal{W}}^{(\eta_\alpha)},d)$.
In order to verify $F^{(\eta_b)}_1\preceq F^{(\eta_b)}_2$
it suffices to show in light of \eqref{pdf.stocle} that
\begin{equation*}
  F^{(\eta_b)}_1(b) = q' \ge P_2(W_b) = F^{(\eta_b)}_2(b) .
\end{equation*}
Similarly we can verify $F^{(\eta_c)}_1\preceq F^{(\eta_c)}_2$.
\end{proof}

\begin{lemma}\label{lem.on.w'}
Let $\mathcal{W}'$ be the subposet induced on $W_a\cup W_d$,
and let $P'_i=P_i-\tilde{Q}'_i-\tilde{Q}''_i$, $i=1,2$, be measures on $\mathcal{W}'$.
Then $P'_1\preceq P'_2$ on $\mathcal{W}'$.
\end{lemma}

\begin{proof}
We find
\begin{align*}
  P'_1(y) &= \begin{cases}
    P_1(\{a\}\cup W_b\cup W_c)-q'-q'' & \mbox{if $y=a$;} \\
    P_1(d) & \mbox{if $y=d$;} \\
    P_1(y) & \mbox{if $y\in W'\setminus\{a,d\}$,}
  \end{cases}
  \\
  P'_2(y) &= \begin{cases}
    P_2(a) & \mbox{if $y=a$;} \\
    P_2(\{d\}\cup W_b\cup W_c)-q'-q'' & \mbox{if $y=d$;} \\
    P_2(y) & \mbox{if $y\in W'\setminus\{a,d\}$.}
  \end{cases}
\end{align*}
We can introduce by the rooted tree $(W',d)$
the distribution function $F'_i$
corresponding to $P'_i$ for $i=1,2$.
Then we obtain
\begin{align*}
  F'_1(a)
  & = P'_1(W_a) = P_1(W_a\cup W_b\cup W_c)-q'-q''
  \\
  & \ge P_2(W_a) = P'_2(W_a) = F'_2(a),
\end{align*}
which implies that $F'_1\preceq F'_2$.
\end{proof}

We define the poset $\mathcal{W}$ of 
$\hat{\mathcal{W}}^{(\eta_b)}$, $\hat{\mathcal{W}}^{(\eta_c)}$
and $\mathcal{W}'$ glued and rooted at $d$,
and introduce a stochastically ordered pair,
$F_1\preceq F_2$, of distribution functions on $\mathcal{W}$ by
\begin{equation*}
  F_i(y) = \begin{cases}
    F^{(\eta_b)}_i(y) & \mbox{if $y\in W_b$;}\\
    F^{(\eta_c)}_i(y) & \mbox{if $y\in W_c$;}\\
    F'_i(y) & \mbox{if $y\in W'\setminus\{d\}$;}\\
    p & \mbox{if $y=d$,}
  \end{cases}
  \quad
  i=1,2,
\end{equation*}
where $p=P_1(S)=P_2(S)$ is the common mass.
By Algorithm~\ref{rp.tree}
we generate a rooted plane tree $(K,1)$ of indices
associated with rooted subtrees of $(W,d)$,
and enumerate the collection
$C(1) = \{\eta_b,\eta_c,\eta_a,\zeta_1,\ldots,\zeta_M\}$
of indices of successor of the root $1$.
We can find a distribution function $\mu$ on $K$
interlaced with $F_1$ and $F_2$ and satisfying $\mu(1)=p$
(which can be obtained by Lemma~\ref{interlace.lem}).
By Algorithm~\ref{map.alg} we form the recursive inverse transforms
$X_{\mu,1}$ and $X_{\mu,2}$ of
\eqref{map.on.w} with $\alpha=1,2$.
By Proposition~\ref{map.prop}
the pair $(X_{\mu,1},X_{\mu,2})$ realizes $F_1$ and $F_2$,
and satisfies $X_{\mu,1}\le X_{\mu,2}$ in $\mathcal{W}$.

\begin{proposition}\label{w.strassen}
There exists a modification
$(\tilde{X}_{\mu,1},\tilde{X}_{\mu,2})$ of the pair
$(X_{\mu,1},X_{\mu,2})$ of recursive inverse transforms
such that it satisfies \eqref{realize}--\eqref{ordered}
with $(X_1,X_2)=(\tilde{X}_{\mu,1},\tilde{X}_{\mu,2})$.
\end{proposition}

\begin{proof}
We find $X_{\mu,1}^{(\eta_b)}(\omega)=b$
and modify it to $\tilde{X}_{\mu,1}^{(\eta_b)}(\omega)=a$ if
\begin{equation*}
  \omega\in
  \left[\mu(\eta_b-),\tilde{F}_1^{(\eta_b)}(a)\right)
  \cup
  \bigcup_{\sigma_i\in C(\eta_b)}\left\{
  \mu^{(\eta_b)}\lfloor\sigma_i\rfloor +
     \left[\tilde{F}_1^{(\sigma_i)}(b),\mu(\sigma_i)\right)
       \right\},
\end{equation*}
where
\begin{align*}
  \gamma &= \max\{\tilde{Q}'_1(a)-F_1(b)+\mu(\eta_b-),0\}; \\
  \tilde{F}^{(\eta_b)}_1(a) &= \mu(\eta_b-)+\tilde{Q}'_1(a)-\gamma; \\
  \tilde{F}^{(\sigma_i)}_1(b)
  &= \mu(\sigma_i) -
  \frac{\gamma[\mu(\sigma_i)-F^{(\eta_b)}_1(u_1^{(\sigma_i)})]}{\mu(\eta_b-)-F^{(\eta_b)}_1(b-)}.
\end{align*}
Similarly we can change from $X_{\mu,1}^{(\eta_c)}(\omega)=c$
to $\tilde{X}_{\mu,1}^{(\eta_c)}(\omega)=a$
on a subset of length $\tilde{Q}''_1(a)$.
The pair $(\tilde{X}_{\mu,1},\tilde{X}_{\mu,2})$
satisfies \eqref{realize}--\eqref{ordered}.
\end{proof}

\bibliographystyle{plain}
\bibliography{monotonicity}

\end{document}